\documentclass[12pt]{article}

\input xy

\xyoption{all}

\usepackage{amssymb,amsbsy,amsthm,amsmath,graphicx,epsfig}
\usepackage{mathabx,times}

\usepackage{sectsty}

\sectionfont{\large\center}

\subsectionfont{\normalsize}

\begin{document}


\newtheorem{thm}{Theorem}[section]
\newtheorem{prop}[thm]{Proposition}
\newtheorem{lem}[thm]{Lemma}
\newtheorem{dfn}[thm]{Definition}
\newtheorem{cor}[thm]{Corollary}

\theoremstyle{remark}

\newtheorem{ques}[thm]{Question}
\newtheorem{rmk}[thm]{Remark}


\renewcommand{\to}{\longrightarrow}

\renewcommand{\d}{\mathrm{d}}
\newcommand{\rme}{\mathrm{e}}
\renewcommand{\a}{\alpha}
\renewcommand{\b}{\beta}
\newcommand{\eps}{\varepsilon}
\newcommand{\g}{\gamma}
\renewcommand{\k}{\kappa}
\newcommand{\s}{\sigma}
\newcommand{\rmH}{\mathrm{H}}
\newcommand{\rmh}{\mathrm{h}}

\newcommand{\A}{\mathcal{A}}
\newcommand{\B}{\mathcal{B}}
\newcommand{\C}{\mathcal{C}}

\newcommand{\G}{\Gamma}
\renewcommand{\O}{\Omega}

\renewcommand{\phi}{\varphi}

\newcommand{\bbZ}{\mathbf{Z}}
\newcommand{\bbN}{\mathbf{N}}
\newcommand{\bbR}{\mathbf{R}}

\newcommand{\R}{\mathcal{R}}

\renewcommand{\cal}[1]{\mathcal{#1}}
\renewcommand{\rm}[1]{\mathrm{#1}}
\renewcommand{\bf}[1]{\mathbf{#1}}
\renewcommand{\t}[1]{\widetilde{#1}}
\newcommand{\ol}[1]{\overline{#1}}
\renewcommand{\hat}[1]{\widehat{#1}}

\newcommand{\fin}{\hspace{\stretch{1}}$\lhd$\\ \vspace{10pt}}
\newcommand{\actson}{\curvearrowright}
\newcommand{\dom}{\mathrm{dom}\,}
\newcommand{\img}{\mathrm{img}\,}
\newcommand{\id}{\mathrm{id}}
\renewcommand{\Vert}{\mathrm{Vert}}
\newcommand{\comp}{\mathrm{comp}}

\title{\large{\textbf{BEHAVIOUR OF ENTROPY UNDER BOUNDED AND INTEGRABLE ORBIT EQUIVALENCE}}}


\author{\textsc{Tim Austin}}



\date{}

\maketitle

\begin{abstract}
Let $G$ and $H$ be infinite finitely generated amenable groups. This paper studies two notions of equivalence between actions of such groups on standard Borel probability spaces.  They are defined as stable orbit equivalences in which the associated cocycles satisfy certain tail bounds.  In `integrable stable orbit equivalence', the length in $H$ of the cocycle-image of an element of $G$ must have finite integral over its domain (a subset of the $G$-system), and similarly for the reverse cocycle.  In `bounded stable orbit equivalence', these functions must be essentially bounded in terms of the length in $G$.  `Integrable' stable orbit equivalence arises naturally in the study of integrable measure equivalence of groups themselves, as introduced recently by Bader, Furman and Sauer.

The main result is a formula relating the Kolmogorov--Sinai entropies of two actions which are equivalent in one of these ways. Under either of these tail assumptions, the entropies stand in a proportion given by the compression constant of the stable orbit equivalence.  In particular, in the case of full orbit equivalence subject to such a tail bound, entropy is an invariant.  This contrasts with the case of unrestricted orbit equivalence, under which all free ergodic actions of countable amenable groups are equivalent.  The proof uses an entropy-bound based on graphings for orbit equivalence relations, and in particular on a new notion of cost which is weighted by the word lengths of group elements.
\end{abstract}

\setcounter{tocdepth}{1}

\tableofcontents

\section{Introduction}

Let $G$ and $H$ be finitely generated discrete groups, and let $T:G\actson (X,\mu)$ and $S:H\actson (Y,\nu)$ be free ergodic actions on standard Borel probability spaces.  A triple such as $(X,\mu,T)$ is called a \textbf{$G$-system}, and similarly for $H$. Let $|\cdot|_G$ and $|\cdot|_H$ be length functions on the two groups given by some choice of finite symmetric generating sets, and let $d_G$ and $d_H$ be the associated right-invariant word metrics.  The generating sets may be written as the unit balls $B_G(e_G,1)$ and $B_H(e_H,1)$ in these metrics.

Recall that a \textbf{stable orbit equivalence} (or \textbf{SOE}) between $(X,\mu,T)$ and $(Y,\nu,S)$ consists of (i) measurable subsets $U\subseteq X$ and $V\subseteq Y$ of positive measure, and (ii) a bi-measurable bijection $\Phi:U\to V$ which satisfies
\[\frac{\mu(\Phi^{-1}A)}{\mu(U)} = \frac{\nu(A)}{\nu(V)} \quad \hbox{for all measurable}\ A \subseteq V\]
and
\[\Phi(T^G(x)\cap U) = S^H(\Phi(x))\cap V \quad \hbox{for $\mu$-a.e.}\ x.\]
If $\mu(U) = \nu(V) = 1$, then $\Phi$ is simply an \textbf{orbit equivalence}, and the systems are said to be \textbf{orbit equivalent}. We often indicate a stable orbit equivalence by $\Phi:(X,\mu,T)\rightarrowtail (Y,\nu,S)$.

A stable orbit equivalence can be described in terms of a pair of maps which convert the $G$-action on the domain to the $H$-action on the target and vice-versa.  For this purpose we make the following definition. An \textbf{$H$-valued partial cocycle over $(X,\mu,T)$} is a pair $(\a,U)$ in which $U\subseteq X$ is measurable and
\[\a:\{(g,x) \in G\times X:\ x \in U\cap T^{g^{-1}}U\}\to H\]
is a measurable function which satisfies the cocycle identity:
\[\a(gk,x) = \a(g,T^kx)\a(k,x) \quad \hbox{whenever}\ g,k \in G\ \hbox{and}\ x \in U\cap T^{k^{-1}}U\cap T^{(gk)^{-1}}U.\]

If $\Phi:(X,\mu,T)\rightarrowtail (Y,\nu,S)$ is a stable orbit equivalence, and $U$ and $V$ are respectively the domain and image of $\Phi$, then $\Phi$ may be described in terms of an $H$-valued partial cocycle $(\a,U)$ over $(X,\mu,T)$ and a $G$-valued partial cocycle $(\b,V)$ over $(Y,\nu,S)$.  They are defined by requiring that
\[\Phi(T^gx) = S^{\a(g,x)}(\Phi(x)) \quad \hbox{whenever}\ x \in U\cap T^{g^{-1}}U\]
and
\[\Phi^{-1}(S^hy) = T^{\b(h,y)}(\Phi^{-1}(y)) \quad \hbox{whenever}\ y \in V\cap S^{h^{-1}}V.\]
These equations specify the cocycles uniquely because the actions are free.  Comparing these equations gives the following relations of inversion between $\a$ and $\b$:
\begin{equation}\label{eq:inversion}
\b(\a(g,x),\Phi(x)) = g \quad \hbox{and} \quad \a(\b(h,y),\Phi^{-1}(y)) = h.
\end{equation}

The category of probability-preserving actions and orbit equivalences has a long history in ergodic theory.  If $G$ and $H$ are amenable then the resulting equivalence relation on systems turns out to be trivial: all free ergodic actions of countable amenable groups are orbit equivalent.  This is the Connes--Feldman--Weiss generalization of Dye's theorem: see~\cite{Dye59,Dye63,ConFelWei81}.  On the other hand, if $G$ is amenable and $(X,\mu,T)$ is a free ergodic $G$-action, then a free ergodic action of another group $H$ can be orbit equivalent to $(X,\mu,T)$ only if $H$ is also amenable: see, for instance,~\cite[Section 4.3]{Zim84}. Among actions of non-amenable groups the relation of orbit equivalence is more complicated.

The generalization to stable orbit equivalence has become important because of its relationship with measure equivalence of groups.  For any countable groups $G$ and $H$, a \textbf{measure coupling} of $G$ and $H$ is a $\sigma$-finite standard Borel measure space $(\O,m)$ together with commuting $m$-preserving actions $G,H \actson \O$ which both have finite-measure fundamental domains.  If a measure coupling exists then $G$ and $H$ are \textbf{measure equivalent}.  This notion was introduced by Gromov in~\cite[Subsection 0.5.E]{Gro93} as a measure-theoretic analog of quasi-isometry.

If $(\O,m)$ is a measure coupling of $G$ and $H$, then one can use fundamental domains for the $G$- and $H$-actions to produce finite-measure-preserving systems for $G$ and $H$ that are stably orbit equivalent.  On the other hand, given a stable orbit equivalence between a $G$-system and an $H$-system, they can be reconstructed into a measure coupling of the groups: see~\cite[Theorem 3.3]{Furman99}, where Furman gives the credit for this result to Gromov and Zimmer.  On account of this correspondence, one can also describe a measure coupling in terms of cocycles over those finite-measure-preserving systems.  This time one obtains cocycles over the whole systems, not just partial cocycles.  In general, it is fairly easy (though not canonical) to extend a partial cocycle to a whole system (this is well-known, but see Proposition~\ref{prop:cocyc-extend} below for a careful proof).

By the aforementioned result of Zimmer, if $G$ is amenable then $H$ can be measure equivalent to $G$ only if $H$ is also amenable.  On the other hand, any amenable group does have actions which are free and ergodic, such as the non-trivial Bernoulli shifts, so the theorem of Connes, Feldman and Weiss shows that any two amenable groups are measure equivalent.

\subsection{Integrability conditions and invariance of entropy}

Recent work of Bader, Furman and Sauer~\cite{BadFurSau13} has introduced a refinement of measure equivalence called `integrable measure equivalence'.  It is obtained by imposing an integrability condition on the cocycles $\a$ and $\b$ that appear in the description of a measure coupling.  Their original results are for hyperbolic groups, but recently this notion has also been studied for amenable groups.  It seems to be a significantly finer relation than measure equivalence.  The growth type of the groups is an invariant, and among groups of polynomial growth the bi-Lipschitz type of the asymptotic cone is an invariant: both of these results are proved in~\cite{Aus--nilpIME}.

The present paper studies stable orbit equivalences which are subject to similar conditions on the integrability or boundedness of their cocycles.  It may be seen as an ergodic theoretic counterpart to the study of integrable measure equivalence, or as a continuation of the study of `restricted orbit equivalences' within ergodic theory.

Because stable orbit equivalences are described in terms of \emph{partial} cocycles, we must be a little careful in the choice of integrability condition to impose.  This paper focuses on two alternatives.  Let $\Phi:(X,\mu,T)\rightarrowtail (Y,\nu,S)$ be an SOE and let $(\a,U)$ and $(\b,U)$ be the partial cocycles which describe it.
\begin{itemize}
 \item We say that $\Phi$ is a \textbf{bounded stable orbit equivalence}, or \textbf{SOE$_\infty$}, if there is a finite constant $C$ such that
\begin{multline*}
|\a(g,x)|_H \leq C|g|_G \quad \hbox{for $\mu$-a.e.}\ x \in U\cap T^{g^{-1}}U\\
\hbox{and} \quad |\b(h,y)|_G \leq C|h|_H \quad \hbox{for $\nu$-a.e.}\ y \in V\cap S^{h^{-1}}V
\end{multline*}
for all $g \in G$ and $h \in H$ (regarding this condition as vacuous if $U\cap T^{g^{-1}}U$ or $V\cap S^{h^{-1}}V$ has measure zero).
\item We say that $\Phi$ is an \textbf{integrable semi-stable orbit equivalence}, or \textbf{SSOE$_1$}, if $(\a,U)$ may be extended to a full cocycle $\s:G\times X\to H$ which satisfies the integrability condition
\[\int_X |\s(g,x)|_H\,\mu(dx) < \infty \quad \forall g \in G,\]
and similarly for $(\b,V)$.  Beware that the extensions of $(\a,U)$ and $(\b,V)$ are not required to satisfy any extended version of~(\ref{eq:inversion}) beyond their original domains.
\end{itemize}

We use the term `semi-stable' for the second possibility because it requires that $\a$ have an extension to all of $G\times X$ which is integrable; it depends on more than just the values taken by $\a$ itself.  We would call $\Phi$ an \textbf{integrable stable orbit equivalence} or \textbf{SOE$_1$} if we required only that
\[\int_{U\cap T^{g^{-1}}U}|\a(g,x)|_H\,\mu(dx) < \infty \quad \forall g \in G.\]
This is formally weaker than both SSOE$_1$ and SOE$_\infty$.  The main result of this paper, Theorem A below, concerns SOE$_\infty$ and SSOE$_1$, but I do not know whether it holds also for SOE$_1$.

We write OE$_\infty$ and OE$_1$ for the special cases of the above notions when $\dom\Phi$ and $\img\Phi$ both have full measure.

In the setting of single probability-preserving transformations, a classical result of Belinskaya~\cite{Bel68} asserts that two transformations $S$ and $T$ are integrably orbit equivalent if and only if $S$ is isomorphic to either $T$ or $T^{-1}$.  Later, several works  studied other notions of `restricted' orbit equivalence for probability-preserving transformations, motived by Kakutani equivalence and Feldman's introduction of loose Bernoullicity: see for instance~\cite{OrnRudWei82} and~\cite{Rud85}. Many of those ideas have been generalized to actions of $\bbZ^d$ for $d\geq 2$ and then to more general amenable groups, culminating in the very abstract formulation of Kammeyer and Rudolph in~\cite{KamRud97,KamRud02}.  For $\bbZ^d$-actions with $d\geq 2$, Fieldsteel and Friedman~\cite{FieFri86} have shown that several natural properties are not invariant under integrable, or even bounded, OE, including discrete spectrum, mixing, and the K property.  However, entropy is an invariant.  Indeed, it is fairly easy to show that OE$_1$ for $\bbZ^d$-systems implies Kakutani equivalence in the sense developed in~\cite{Kat77,delJRud84} (see Section~\ref{sec:Ab} below), and those works include the result that entropy is invariant under Kakutani equivalence.

The present work extends this last conclusion to SOE$_\infty$ and SSOE$_1$ and to general discrete amenable groups.

\vspace{7pt}

\noindent\textbf{Theorem A}\quad \emph{Suppose that $G$ and $H$ are amenable, that $(X,\mu,T)$ and $(Y,\nu,S)$ are as above, and that $\Phi:U\to V$ is either a SOE$_\infty$ or a SSOE$_1$.  Then
\[\mu(U)^{-1}\rmh(\mu,T) = \nu(V)^{-1}\rmh(\nu,S).\]}

\vspace{7pt}

\begin{rmk}
It suffices to assume that only one of $G$ and $H$ is amenable, since the existence of the stable orbit equivalence then implies that the other is too. \fin
\end{rmk}

%
%
%
%

\subsection{Derandomization of orbit equivalences}

We prove Theorem A in two parts.

The first part handles the case of the Euclidean lattices $\bbZ^d$.  In this case, our various notions of stable orbit equivalence turn out to imply Kakutani equivalence, one of the more classical notions of restricted orbit equivalence.

\vspace{7pt}

\noindent\textbf{Theorem B}\quad \emph{Suppose that $(X,\mu,T)$ is a $\bbZ^d$-system and $(Y,\nu,S)$ is a $\bbZ^D$-system.  If they are SOE$_\infty$ then they are SSOE$_1$, and if they are SSOE$_1$ then $d=D$ and they are Kakutani equivalent.}

\vspace{7pt}

This will be proved in Section~\ref{sec:Ab}. Theorem A follows for these groups because it is known how entropy transforms under Kakutani equivalences of $\bbZ^d$-systems~\cite{delJRud84}.

Moreover, a fairly standard construction (see Section~\ref{sec:finite-ind}) allows one to pass between groups and their finite-index subgroups, and so from Theorem B we can deduce Theorem A for all finitely generated, virtually Abelian groups.

In the second part of the proof, all remaining cases are deduced from a result that we call `OE derandomization'.  It asserts that if a SOE between two systems is a SSOE$_1$ or SOE$_\infty$, then it is lifted from a SOE between two factor systems having arbitrarily low entropy (that is, having `arbitrarily little randomness').  Curiously, this result seems to require \emph{super}-linear growth of the acting groups, and so it cannot be used to prove Theorem A for virtually cyclic groups.  Thus we need both approaches to prove Theorem A in general.

A simpler version of derandomization can be observed among arbitrary integrable cocycles from a system to a group.  We state (and later prove) this result first, as motivation for the orbit-equivalence result that we need.

\vspace{7pt}

\noindent\textbf{Theorem C (Cocycle derandomization)}\quad \emph{Let $G$ be amenable and have super-linear growth, let $(X,\mu,T)$ be a $G$-system, and let $\s:G\times X\to H$ be an integrable cocycle over $T$.  For any $\eps > 0$ there is a cocycle $\tau$ cohomologous to $\s$ over $T$ such that the factor of $(X,\mu,T)$ generated by $\tau$ has entropy less than $\eps$.}

\vspace{7pt}

For general cocycles $G\times X\to H$, not necessarily arising from an SOE, boundedness and integrability are defined in Subsection~\ref{subs:syscoc}.  For cocycles which are not partial, boundedness implies integrability, so Theorem C applies in particular to all bounded cocycles.

Now suppose that $\Phi:(X,\mu,T)\rightarrowtail (Y,\nu,S)$ is an SOE.  If $U\subseteq \dom\Phi$ is measurable and has positive measure, then the restriction $\Phi|U$ still defines an SOE, different from $\Phi$ in that the domain and image have been made smaller.  The \textbf{restriction} of $\Phi$ to $U$ is always understood as a SOE in this way.

\vspace{7pt}

\noindent\textbf{Theorem D (Orbit-equivalence derandomization})\quad \emph{Let $G$ be amenable and have super-linear growth, and let $\Phi:(X,\mu,T)\rightarrowtail (Y,\nu,S)$ be an SOE which is either a SSOE$_1$ or a SOE$_\infty$.  Let $\eps > 0$.  Then there are
\begin{itemize}
\item a measurable subset $U \subseteq \dom\Phi$ with $\mu(U) > 0$,
\item factor maps $\pi:(X,\mu,T)\to (X',\mu',T')$ and $\xi:(Y,\nu,S)\to (Y',\nu',S')$ whose target systems are still free,
\item and a SOE $\Phi':(X',\mu',T') \rightarrowtail (Y',\nu',S')$
\end{itemize}
such that
\begin{enumerate}
\item $\rmh(\mu',T') < \eps$ and
\item the following diagram commutes:
\begin{center}
$\phantom{i}$\xymatrix{
(X,\mu,T) \ar@{>->}^{\Phi|U}[r]\ar_\pi[d] &(Y,\nu,S) \ar^\xi[d]\\
(X',\mu',T') \ar@{>->}_{\Phi'}[r] & (Y',\nu',S').
}
\end{center}
\end{enumerate}}

\vspace{7pt}

%
%
%
%
%
%
%
%

For groups of super-linear growth, Theorem A is deduced from Theorem D in Section~\ref{sec:AfromC}.  Then Sections~\ref{sec:graphings} and~\ref{sec:low-cost} develop some more technical results, before Theorems C and D are proved in Section~\ref{sec:BandC}.  Those technical results include a new notion of cost for a graphing on a Borel orbit equivalence relation which takes into account the word lengths of different group elements, and may be of independent interest.  It appears in Definition~\ref{dfn:cost}.

\subsubsection*{Acknowledgements}

This paper emerged from an ongoing collaboration with Uri Bader, Lewis Bowen, Alex Furman and Roman Sauer.  I am also grateful to Oded Regev, Damien Gaboriau and Brandon Seward for some useful references.  Finally, I thank the anonymous reviewer for suggestions which clarified various technical steps in the proofs.

\section{Background from ergodic theory}

\subsection{Systems and partial cocycles}\label{subs:syscoc}

All measure spaces in this paper are standard Borel and $\s$-finite.  Most are probability spaces.  Measure spaces are denoted by pairs such as $(X,\mu)$; the $\s$-algebra of this space will be denoted by $\B_X$ when it is needed.

An \textbf{observable} on a measure space $(X,\mu)$ is a measurable function $\phi$ from $X$ to a countable set, and a \textbf{partial observable} on $(X,\mu)$ is a pair $(\phi,U)$ consisting of a measurable subset $U \subseteq X$ and a measurable function $\phi$ from $U$ to a countable set.  

A \textbf{$G$-system} is a triple $(X,\mu,T)$ consisting of a standard Borel probability space $(X,\mu)$ and a $\mu$-preserving measurable action $T$ of $G$ on that space.  It is \textbf{free} if the orbit-map $g\mapsto T^gx$ is injective for $\mu$-a.e. $x$.  The Borel orbit equivalence relation of this action is denoted by $\R_T$. We assume standard definitions and results about orbit equivalence and cocycles over such systems: see, for instance,~\cite[Section 4.2]{Zim84}.

Conventions seem a little less settled in relation to stable orbit equivalence, and I do not know of a canonical reference.  It appears most often in connection with measure equivalence of groups, such as in~\cite{Furman99,Furman11,Gab02,Gab05,Shalom04}.  The present paper uses slightly different conventions, since our interest is in the systems and not just the groups.  But I have followed~\cite[Section 2]{Furman99} where possible.

If $H$ is another discrete group, then an \textbf{$H$-valued partial cocycle} over $(X,\mu,T)$ is a pair $(\a,U)$ in which $U \subseteq X$ is measurable and
\[\a:\{(g,x):\ x \in U\cap T^{g^{-1}}U\}\to H\]
is a measurable function satisfying the cocycle identity
\[\a(gk,x) = \a(g,T^kx)\a(k,x) \quad \hbox{whenever}\ g,k \in G\ \hbox{and}\ x \in U\cap T^{k^{-1}}U\cap T^{(gk)^{-1}}U.\]
This reduces to the usual notion of a cocycle if $U = X$. We sometimes write $\a^g$ for the function
\[\a(g,\cdot):U\cap T^{g^{-1}}U\to H,\]
and if $x \in U$ then we write $\a_x$ for the function
\[\a(\cdot,x):\{g\in G:\ T^gx \in U\}\to H.\]
A partial cocycle $(\a,U)$ is \textbf{non-trivial} if $\mu(U) > 0$. Two partial cocycles are considered equal if their sets are equal modulo $\mu$ and their functions agree $\mu$-a.e.

If $(\a,U)$ is a partial cocycle over $(X,\mu,T)$ and $V \subseteq U$ is measurable, then the \textbf{restriction} of $(\a,U)$ to $V$ is the partial cocycle $(\a_{|V},V)$ where $\a_{|V}$ is the restriction of the map $\a$ to the set $\{(g,x):\ x \in V\cap T^{g^{-1}}V\}$.  To lighten notation we sometimes write this restriction as $(\a,V)$.

A cocycle $\a$ or partial cocycle $(\a,U)$ is \textbf{bounded} if there is a finite constant $C$ such that
\[|\a(g,x)|_H \leq C|g|_G \quad \hbox{for $\mu$-a.e.}\ x \in U\cap T^{g^{-1}}U,\ \hbox{for all}\ g \in G.\]
A cocycle $\a$ (not partial) is \textbf{integrable} if
\[\int |\a(g,x)|_H\,\mu(dx) < \infty \quad \forall g \in G.\]
Clearly a bounded cocycle is integrable. These usages are consistent with the definitions of SOE$_\infty$ and SSOE$_1$ in the Introduction.

If $\a:G\times X\to H$ is a cocycle, then it is bounded if and only if each of the finitely many functions
\[|\a(s,\cdot)|_H, \quad s \in B_G(e_G,1),\]
is essentially bounded on $X$.  The forward implication here is immediate, and the reverse follows by writing a general element of $G$ as $g = s_\ell\cdots s_1$ with $\ell = |g|_G$ and $s_1,\dots,s_\ell \in B_G(e_G,1)$, and then using the cocycle identity
\begin{equation}\label{eq:cocyc-prod}
\a(g,x) = \a(s_\ell,T^{s_{\ell-1}\cdots s_1}x)\cdots \a(s_1,x).
\end{equation}
However, we cannot argue this way for a partial cocycle $(\a,U)$, since the factors on the right-hand side of~(\ref{eq:cocyc-prod}) may not all be defined for arbitrary $x \in U\cap T^{g^{-1}}U$.  This is why we use the definition of boundedness given above.

\subsection{Entropy}

Let $(X,\mu)$ be a probability space.  If $\mu$ is atomic, then its Shannon entropy is
\[\rmH(\mu):= -\sum_{x \in X}\mu\{x\}\log \mu\{x\},\]
with the usual interpretation $0\log 0 := 0$.

If $\phi:X\to A$ is an observable, then its Shannon entropy is
\[\rmH_\mu(\phi) := \rmH(\phi_\ast\mu).\]
If $U \subseteq X$ is measurable, then its Shannon entropy is defined to be that of the indicator function $1_U$: more explicitly,
\[\rmH_\mu(U) := -\mu(U)\log \mu(U) - \mu(X\setminus U)\log \mu(X\setminus U).\]

If $(X,\mu)$ is a probability space and $(\phi,U)$ is a partial observable on it, then the Shannon entropy of $(\phi,U)$ is defined to be
\[\rmH_\mu(\phi;U) := \rmH_\mu(U) + \mu(U)\cdot \rmH_{\mu_{|U}}(\phi),\]
where $\mu_{|U}$ is the measure $\mu$ conditioned on $U$: that is,
\[\mu_{|U}(V) := \mu(V\cap U)/\mu(U).\]
If $\mu(U) = 0$, then we set $\rmH_\mu(\phi;U) = 0$ by convention.  If $(\phi,U)$ is a partial observable and $V\subseteq U$ is measurable, then we abbreviate $\rmH_\mu(\phi|V;V)$ to just $\rmH_\mu(\phi;V)$.

Observe that, if $(\phi,U)$ is a partial observable and $\ast$ is an abstract point outside the range of $\phi$, then we can define a new observable $\phi^\ast$ by
\[\phi^\ast(x) = \left\{\begin{array}{ll}\phi(x) &\quad \hbox{if}\ x \in U\\ \ast & \quad \hbox{if}\ x \in X\setminus U,\end{array}\right.\]
and we obtain $\rmH_\mu(\phi;U) = \rmH_\mu(\phi^\ast)$.

Now let $G$ be a discrete amenable group with a F\o lner sequence $(F_n)_{n\geq 1}$, and let $(X,\mu,T)$ be a $G$-system.  If $\phi:X\to A$ is an observable and $F \subseteq G$ is finite, let
\[\phi^F:= (\phi\circ T^g)_{g \in F}:X\to A^F.\]
The factor generated by $\phi$ is the $\s$-algebra of subsets of $X$ generated by the level-sets of $\phi$ and all their images under $T^g$, $g \in G$.  If $(\phi,U)$ is a partial observable, then the factor it generates is defined to be the factor generated by $\phi^\ast$, the new observable constructed above.

As is standard, the \textbf{Kolmogorov--Sinai} (`\textbf{KS}') \textbf{entropy} of the system $(X,\mu,T)$ and observable $\phi$ is
\[\rmh(\mu,T,\phi) := \lim_{n\to\infty}\frac{1}{|F_n|}\rmH_\mu(\phi^{F_n}).\]
This may be calculated using any F\o lner sequence for $G$.  Then the KS entropy of $(X,\mu,T)$ is the supremum of $\rmh(\mu,T,\phi)$ over all observables $\phi$.  It is denoted by $\rmh(\mu,T)$.  By the Kolmogorov--Sinai theorem, the quantity $\rmh(\mu,T,\phi)$ is always equal to the KS entropy of the factor of $(X,\mu,T)$ generated by $\phi$.

The subadditivity of Shannon entropy has the immediate consequence
\[\rmh(\mu,T,\phi) \leq \rmH_\mu(\phi).\]
We extend this to a partial observable $(\phi,U)$ by defining $\rmh(\mu,T,(\phi,U))$ to be the KS entropy of the factor generated by $(\phi,U)$.  By writing this in terms of the extended observable $\phi^\ast$, we immediately obtain also
\begin{equation}\label{eq:Shannon-bound}
\rmh(\mu,T,(\phi,U)) \leq \rmH_\mu(\phi;U).
\end{equation}

The following useful estimate may be well-known, but I have not found a reference.  It was shown to me by Alex Furman.

\begin{lem}\label{lem:Fur}
Let $|\cdot|_G$ be a length function on $G$ corresponding to a finite symmetric generating set.  For every $\eps > 0$ there exists $C_\eps < \infty$ such that the following holds. If $p_g$ is a value in $[0,1]$ for every $g \in G\setminus \{e_G\}$, then
\[\sum_{g \neq e_G}[-p_g\log p_g - (1-p_g)\log(1-p_g)\big] \leq C_\eps\sum_{g \neq e_G}|g|_Gp_g + \eps.\]
In particular, if $(X,\mu)$ is a probability space and $\a:X\to G\setminus\{e_G\}$ is an observable, then
\[\rmH_\mu(\a) \leq C_\eps\int |\a(x)|_G\,\mu(dx) + \eps\]
\end{lem}

\begin{proof}
First, Markov's Inequality gives
\[|\{g\in G\setminus \{e_G\}:\, p_g \geq 1/2\}| \leq 2\sum_{g \neq e_G}p_g.\]
On the other hand, if $p_g \leq 1/2$ then
\[-(1-p_g)\log(1-p_g) \leq -p_g\log p_g.\]
We may therefore bound the desired sum as follows:
\begin{eqnarray*}
&& \sum_{g \neq e_G}[-p_g\log p_g - (1-p_g)\log(1-p_g)\big]\\ && \leq 2\sum_{g \neq e_G}(-p_g\log p_g) + \log 2\cdot |\{g\in G\setminus \{e_G\}:\, p_g \geq 1/2\}|\\
&&\leq 2\sum_{g \neq e_G}(-p_g\log p_g) + 2\log 2\sum_{g\neq e_G}|g|_G p_g.
\end{eqnarray*}
It therefore suffices to prove that $\sum_{g\neq e_G} (-p_g\log p_g)$ may be bounded in terms of $\sum_{g\neq e_G} |g|_Gp_g$ in the desired way.

Next, since $G$ is finitely generated, there is a finite constant $c$ such that
\[|B_G(e_G,n)| \leq c^n \quad \forall n\geq 0.\]
For each $n\geq 1$, let
\[q_n := \sum_{|g|_G = n}p_g.\]
Let $\mu_g := p_g/q_{|g|_G}$ for all $g \in G$, interpreting this as $0$ if $q_{|g|_G} = 0$. Provided ${q_n \neq 0}$, the tuple $(\mu_g)_{|g|_G = n}$ is a probability distribution on the $|\cdot|_G$-sphere ${\{|g|_G = n\}}$.  From this fact we derive the estimate
\begin{eqnarray*}
\sum_{|g|_G = n}(-p_g\log p_g) &=& \sum_{|g|_G = n}(-(\mu_gq_n)\log (\mu_gq_n))\\
 &=& q_n\rmH\big((\mu_g)_{|g|_G = n}\big) - q_n\log q_n\\
 &\leq& q_n\log|B_G(e_G,n)| - q_n\log q_n\\
 &\leq& cnq_n - q_n\log q_n.
\end{eqnarray*}
whenever $q_n \neq 0$.

Finally, some elementary calculus gives
\[ - t \log t \leq mt + \rme^{-m-1} \quad \hbox{for any}\ t,m>0.\]
Let $k > 0$ be large and fixed, and for each $n\geq 1$ apply this bound with $t := q_n$ and $m := kn$.  It gives
\[\sum_{n\geq 1}(-q_n\log q_n) \leq k\sum_{n\geq 1}nq_n + \sum_{n\geq 1}\rme^{-kn-1}.\]
Combining this with the previous estimate, we obtain
\[\sum_{g\neq e_G}(-p_g\log p_g) = \sum_{n\geq 1}\sum_{|g|_G = n}(-p_g\log p_G) \leq (c + k)\sum_{n\geq 1}nq_n + \sum_{n\geq 1}\rme^{-kn-1}.\]
By choosing $k$ large enough we may make the last term here less than $\eps$, so this completes the proof of the first inequality.

We obtain the second part of the lemma by applying that first inequality to the values $p_g := \mu\{\a = g\}$.
\end{proof}

\begin{cor}\label{cor:cocyc-part-ent}
Suppose that $G$ and $H$ are finitely generated groups, that $(X,\mu,T)$ is a $G$-system, and that $\a:G\times X\to H$ is an integrable cocycle over $T$.  For every $g \in G$ and every $\eps > 0$ there exists $\delta > 0$ for which the following holds: for any measurable $U \subseteq X$,
\[\hbox{if}\quad \mu(U) < \delta\quad \hbox{then}\quad \rmH_\mu(\a^g;U) < \eps.\]
\end{cor}

\begin{proof}
First, any sufficiently small $\delta$ satisfies
\[\mu(U) < \delta \quad \Longrightarrow \quad \rmH_\mu(U) < \eps/3.\]

On the other hand, since $\a^g$ is integrable, for any $\eta > 0$ there is a $\delta > 0$ such that
\[\mu(U) < \delta \quad \Longrightarrow \quad \int_U |\a(g,x)|_H\,\mu(d x) < \eta.\]
By a special case of Lemma~\ref{lem:Fur}, we may choose $C < \infty$ so that this turns into
\[\rmH_{\mu_{|U}}(\a^g) \leq C\eta/\mu(U) + 1,\]
where $C$ does not depend on the value of $\eta$.

Combining these estimates gives
\[\mu(U) < \delta \quad \Longrightarrow \quad \rmH_\mu(\a^g;U) < \eps/3 + C\eta + \mu(U) < \eps/3 + C\eta + \delta.\]
Choosing $\eta < \eps/3C$ and then ensuring that $\delta < \eps/3$, this completes the proof.
\end{proof}

\section{Some preliminaries on stable orbit equivalence and cocycles}

\subsection{Generated factors}

Given a partial cocycle $(\a,U)$ over $(X,\mu,T)$, the factor that it \textbf{generates} is the smallest factor which contains $U$ and with respect to which all of the partial observables
\[(\a^g,U\cap T^{g^{-1}}U)\]
are measurable.  More explicitly, it is generated by the sets
\[U_{g,h} := \{x \in U\cap T^{g^{-1}}U:\ \a(g,x) = h\}\]
for $g \in G$ and $h \in H$, together with all their $T$-images.  Notice that if we fix $g$ and let $h$ vary over $H$ then the sets $U_{g,h}$ constitute a measurable partition of $U\cap T^{g^{-1}}U$.

%
%
%
%

Now suppose that $\Phi:(X,\mu,T) \rightarrowtail (Y,\nu,S)$ is an SOE, let $U:= \dom\Phi$ and let $V:= \img\Phi$.  The \textbf{compression} of $\Phi$ is the constant
\[\comp(\Phi) := \frac{\nu(V)}{\mu(U)}.\]
Let $(\a,U)$ be the partial cocycle associated to $\Phi$, and $(\b,V)$ that associated to $\Phi^{-1}$.  Let $U_{g,h}$ be the sets defined above for the partial cocycle $(\a,U)$, and let $V_{h,g}$ be their counterparts for $(\b,V)$.  For any $g \in G$ and $h \in H$, the relation between $\a$ and $\b$ implies that
\begin{align*}
&x\in U\cap T^{g^{-1}}U\ \hbox{and}\ \a(g,x) = h \\ &\Longleftrightarrow \quad x\in U\cap T^{g^{-1}}U\ \hbox{and}\ \Phi(T^gx) = S^h\Phi(x)\\
&\Longleftrightarrow \quad \Phi(x) \in V\cap S^{h^{-1}}V\ \hbox{and}\ \b(h,\Phi(x)) = g.
\end{align*}
Therefore $\Phi(U_{g,h}) = V_{h,g}$ for all $g$ and $h$, and
\begin{equation}\label{eq:Ugh}
\Phi(T^gA) = S^h\Phi(A) \quad \forall A\subseteq U_{g,h}.
\end{equation}

The next lemma is the first definite step in the direction of Theorem D.

\begin{lem}\label{lem:pre-C}
Let $\A$ be a factor of $(X,\mu,T)$ with respect to which $(\a,U)$ is measurable.  Then there are
\begin{itemize}
\item a factor map $\pi:(X,\mu,T)\to (X',\mu',T')$ which generates $\A$ modulo $\mu$,
 \item another factor map $\xi:(Y,\nu,S)\to (Y',\nu',S')$,
\item and a SOE $\Phi':(X',\mu',T')\to (Y',\nu',S')$
\end{itemize}
such that the diagram
\begin{center}
$\phantom{i}$\xymatrix{
(X,\mu,T) \ar@{>->}^{\Phi}[r]\ar_\pi[d] &(Y,\nu,S) \ar^\xi[d]\\
(X',\mu',T') \ar@{>->}_{\Phi'}[r] & (Y',\nu',S')
}
\end{center}
commutes in the following sense: $\dom \Phi = \pi^{-1}(\dom \Phi')$, and
\[\Phi'\circ (\pi|\dom \Phi) = \xi\circ \Phi\]
almost surely on this set.
\end{lem}

\begin{proof}
Let $\A_0 := \A\cap U$, let $\C_0 := \Phi[\A_0]$, and let $\C$ be the factor of $(Y,\nu,S)$ generated by $\C_0$.

We now show that $\C\cap V = \C_0$.  The inclusion $\supseteq$ is obvious.  For the reverse, let us show that $\C$ is generated as a $\s$-algebra by a family of sets whose intersections with $V$ are all members of $\C_0$.  In particular, let $A \in \A_0$, let $C := \Phi(A)$, and let $h \in H$: we will show that $D := S^{h^{-1}}C\cap V$ still lies in $\C_0$. Since $C \subseteq V$, we have $D \subseteq S^{h^{-1}}V \cap V$. This right-hand set is partitioned into the subsets $V_{h,g} = \Phi(U_{g,h})$, $g \in G$, and these are all members of $\Phi[\A_0] = \C_0$ by our assumption that $(\a,U)$ is $\A$-measurable.  Therefore
\begin{multline*}
D = \bigcup_{g \in G}D\cap V_{h,g} = \bigcup_{g \in G}S^{h^{-1}}C\cap V_{h,g} = \bigcup_{g \in G}S^{h^{-1}}\big(C\cap S^hV_{h,g}\big) \\= \bigcup_{g\in G}S^{h^{-1}}\big(\Phi(A \cap T^gU_{g,h})\big) = \bigcup_{g \in G}\Phi(T^{g^{-1}}A\cap U_{g,h}),
\end{multline*}
using~(\ref{eq:Ugh}) for the fourth and fifth equalities. This is explicitly a member of $\Phi[\A_0] = \C_0$, as required.

Next, since $(X,\mu)$ and $(Y,\nu)$ are standard Borel, we may choose factor maps $\pi:(X,\mu,T)\to (X',\mu',T')$ and $\xi:(Y,\nu,S)\to (Y',\nu',S')$ which generate $\A$ modulo $\mu$ and $\C$ modulo $\nu$, respectively.  Since $U \in \A_0 \subseteq \A$ and $V \in \C_0 \subseteq \C$, there are measurable subsets $U' \subseteq X'$ and $V' \subseteq Y'$ such that $U = \pi^{-1}U'$ and $V = \xi^{-1}V'$  modulo negligible sets.  Since $\A\cap U = \A_0$ and $\C\cap V = \C_0$, it follows that $\A_0 = \pi^{-1}[\B_{U'}]$ modulo $\mu$ and $\C_0 = \xi^{-1}[\B_{V'}]$ modulo $\nu$, respectively.  Therefore the set-mapping
\[\Phi[\cdot]:\A_0 \to \C_0\]
defines a measure-algebra equivalence from $\B_{U'}$ modulo $\mu'$ to $\B_{V'}$ modulo $\nu'$.  Since $U'$ and $V'$ are standard Borel, this arises from a measurable bijection $\Phi':U'\to V'$.  Now a simple diagram-chase shows that this fits into the desired commutative diagram.
\end{proof}

\subsection{Extensions of partial cocycles and of systems}

As remarked in the introduction, there is a close relationship between stable orbit equivalence of systems and measure equivalence of the acting groups:~\cite[Theorem 3.3]{Furman99}.  The main results of the present paper concern entropy, which is a property of the systems rather than the groups, so our point of view emphasizes the former.  However, some of the results we need are already known in the study of measure equivalence, including most of those in this subsection.

The first such result is a general procedure for extending a partial cocycle to a full cocycle.  This fact can easily be extracted from the proof of the equivalence between stable orbit equivalence and measure equivalence, but for completeness we include a proof purely in terms of cocycles.

\begin{prop}\label{prop:cocyc-extend}
If $(X,\mu,T)$ is an ergodic $G$-system and $(\a,U)$ is a non-trivial $H$-valued partial cocycle over it, then there is a cocycle $\s:G\times X\to H$ such that $\a = \s_{|U}$.

If $\tau:G\times X \to H$ is another cocycle satisfying $\a = \tau_{|U}$, then $\s$ and $\tau$ are cohomologous over $(X,\mu,T)$.
\end{prop}

\begin{proof}
\emph{Part 1.}\quad Let us enumerate $G = \{g_1=e_G,g_2,g_3,\ldots\}$.  Since $\mu(U) > 0$ and the system is ergodic, we have
\[\mu\Big(\bigcup_{g\in G}T^{g^{-1}}U\Big) = 1:\]
that is, $U$ meets almost every $T$-orbit. Therefore for a.e. $x \in X$ there is a minimal $n \in \bbN$ such that $T^{g_n}(x) \in U$.  This choice of $g_n$ defines a measurable function $\g:X\to G$ such that $T^{\g(x)}(x) \in U$ for a.e. $x$.  We call it the \textbf{$U$-return map}.

We now define $\s(g,x)$ by
\begin{itemize}
 \item moving both $x$ and $T^gx$ into the set $U$ using the $U$-return map, and then
\item taking the value of $\a$ that connects those two new points.
\end{itemize}
To be precise, this means that
\[\s(g,x) := \a\big(\g(T^gx)g\g(x)^{-1},T^{\g(x)} (x)\big).\]
To see that this is well-defined, observe that the definition of $\g$ gives $T^{\g(x)}(x) \in U$ and also
\[T^{\g(T^gx)g\g(x)^{-1}}\big(T^{\g(x)}(x)\big) = T^{\g(T^gx)}(T^gx) \in U,\]
and so
\[T^{\g(x)}(x) \in U\cap T^{(\g(T^gx)g\g(x)^{-1})^{-1}}U = \dom\big(\a\big(\g(T^gx)g\g(x)^{-1},\cdot\big)\big).\]

A simple check using the cocycle equation for $\a$ shows that the new map $\s$ also satisfies the cocycle equation:
\begin{align*}
\s(gk,x) &= \a\big(\g(T^{gk}x)gk\g(x)^{-1},T^{\g(x)} (x)\big)\\
&= \a\big(\g(T^{gk}x)g\g(T^kx)^{-1}\cdot \g(T^kx)k\g(x)^{-1},T^{\g(x)} (x)\big)\\
&= \a\big(\g(T^{gk}x)g\g(T^kx)^{-1},T^{\g(T^kx)} (T^kx)\big)\cdot \a\big(\g(T^kx)k\g(x)^{-1},T^{\g(x)} (x)\big)\\
&= \s(g,T^kx)\s(k,x).
\end{align*}

Lastly, $\s$ extends $\a$, because if $x \in U\cap T^{g^{-1}}U$ then $\g(x) = \g(T^gx) = e_G$ (recalling that we put $e_G$ first in our enumeration of $G$), and so
\[\s(g,x) = \a\big(\g(T^{g}x)g\g(x)^{-1},T^{\g(x)} (x)\big) = \a(g,x).\]

\vspace{7pt}

\emph{Part 2.}\quad If $\tau$ is another cocycle for which $\tau_{|U} = \a$, then the cocycle equation for $\tau$ gives
\begin{align*}
\s(g,x) &= \a\big(\g(T^gx)g\g(x)^{-1},T^{\g(x)} (x)\big)\\
&= \tau\big(\g(T^gx)g\g(x)^{-1},T^{\g(x)} (x)\big)\\
&= \tau(\g(T^gx),T^gx)\cdot \tau(g,x)\cdot \tau(\g(x)^{-1},T^{\g(x)}(x))\\
&= \tau(\g(T^gx),T^gx)\cdot \tau(g,x)\cdot \tau(\g(x),x)^{-1}\\
&= \eta(T^g x)^{-1}\cdot \tau(g,x)\cdot \eta(x),
\end{align*}
where $\eta(x) := \tau(\g(x),x)^{-1}$ is a measurable function from $X$ to $H$.  So $\s$ is manifestly cohomologous to $\tau$.
\end{proof}

\begin{rmk}
If the partial cocycle $(\a,U)$ satisfies an assumption of boundedness or integrability, then Proposition~\ref{prop:cocyc-extend} gives no guarantee that its extension $\s$ satisfies the same assumption.  We must therefore by quite careful in how we apply this proposition to the study of SSOE$_1$ or SOE$_\infty$. In the case of SSOE$_1$, an integrable extended cocycle is guaranteed by definition, but we sometimes need to perform some other manipulations on a cocycle first and then apply Proposition~\ref{prop:cocyc-extend}, so care is still necessary. \fin
\end{rmk}

Now suppose that $\Phi:(X,\mu,T) \rightarrowtail (Y,\nu,S)$ is a stable orbit equivalence from a free ergodic $G$-system to a free ergodic $H$-system. Let $(\a,U)$ and $(\b,V)$ be the partial cocycles that describe $\Phi$ and $\Phi^{-1}$. We now use Proposition~\ref{prop:cocyc-extend} to construct a kind of `common extension' of the two systems $(X,\mu,T)$ and $(Y,\nu,S)$.

This construction can be carried out starting from either $(X,\mu,T)$ or $(Y,\nu,S)$.  We begin by using the former.  First, apply Proposition~\ref{prop:cocyc-extend} to obtain a cocycle $\hat{\a}:G\times X\to H$ such that $\a = \hat{\a}_{|U}$. Now let $\hat{X} := X\times H$ and let $\hat{\mu}$ be the $\s$-finite measure on this space which is the product of $\mu$ and counting measure.  We define an infinite-measure-preserving action $\hat{T}$ of $G\times H$ on $(\hat{X},\hat{\mu})$ by setting
\[\hat{T}^{(g,h)}(x,k) := (T^gx,\hat{\a}(g,x) kh^{-1}) \quad \hbox{for}\ g \in G\ \hbox{and}\ h \in H.\]
The resulting system $(\hat{X},\hat{\mu},\hat{T})$ is ergodic.  Indeed, if $A \subseteq \hat{X}$ its invariant, then the action of $H$ on the vertical fibres of $\hat{X}$ forces $A$ to have been lifted from $X$, but an invariant set lifted from $X$ must be negligible or co-negligible because $T$ is ergodic.

Starting with $(Y,\nu,S)$, the analogous construction uses an extension $\hat{\b}$ of $\b$ to define a $(G\times H)$-action $\hat{S}$ on $\hat{Y} := Y\times G$ that preserves the product $\hat{\nu}$ of $\nu$ and counting measure.

\begin{lem}\label{lem:inf-meas-iso}
The infinite-measure-preserving $(G\times H)$-systems $(\hat{X},\hat{\mu},\hat{T})$ and $(\hat{Y},\hat{\nu},\hat{S})$ are isomorphic, up to changing the measures by a constant multiple.
\end{lem}

\begin{proof}
Define
\[\hat{U} := U\times \{e_H\} \subseteq \hat{X} \quad \hbox{and}\quad \hat{V} := V\times \{e_G\} \subseteq \hat{Y},\]
and let $\Phi_1:\hat{U}\to \hat{V}$ be the map that results from the obvious identification of $\hat{U}$ with $U$ and $\hat{V}$ with $V$.

The idea is that $\Phi_1$ should be a `part' of the required isomorphism, and now $(G\times H)$-equivariance tells us how to extend it.  Thus, for $(x,k) \in \hat{X}$, choose $(g,h) \in G\times H$ so that $\hat{T}^{(g,h)}(x,k) \in \hat{U}$ (we may do this almost surely by the ergodicity of $\hat{T}$), and let
\[\hat{\Phi}(x,k) := \hat{S}^{(g^{-1},h^{-1})}\big(\Phi_1(\hat{T}^{(g,h)}(x,k))\big).\]
We must check that this is well-defined.  Suppose that $(g',h') \in G\times H$ also satisfies $\hat{T}^{(g',h')}(x,k) \in \hat{U}$, and let $(g_1,h_1) := (g'g^{-1},h'h^{-1})$.  The assumption that $\hat{T}^{(g,h)}(x,k) \in \hat{U}$ is equivalent to $T^gx \in U$ and $\hat{\a}(g,x)kh^{-1} = e_G$, and similarly for $(g',h')$.  Combining these relations with the cocycle equation for $\hat{\a}$, we obtain
\[e_G = \hat{\a}(g',x)\cdot k\cdot (h')^{-1} = \a(g_1,T^gx)\hat{\a}(g,x)kh^{-1}h_1^{-1} = \a(g_1,T^gx)h_1^{-1}.\]
Hence $h_1 = \a(g_1,T^gx)$, and so also $g_1 = \b(h_1,\Phi(T^gx))$, by~(\ref{eq:inversion}).  From this we deduce that
\begin{align*}
\Phi_1(\hat{T}^{(g',h')}(x,k)) &= \big(\Phi(T^{g'}x),e_G\big) = \big(\Phi(T^{g_1}(T^gx)),e_G\big)\\ &= \big(S^{\a(g_1,T^gx)}(\Phi(T^gx)),e_G\big) = \big(S^{h_1}(\Phi(T^gx)),e_G\big)\\  &= \big(S^{h_1}(\Phi(T^gx)),\b(h_1,\Phi(T^gx))e_Gg_1^{-1}\big)\\ &= \hat{S}^{(g_1,h_1)}(\Phi_1(\hat{T}^{(g,h)}(x,k)).
\end{align*}
Therefore
\begin{multline*}
\hat{S}^{(g',h')^{-1}}\big(\Phi_1(\hat{T}^{(g',h')}(x,k))\big) = \hat{S}^{(g',h')^{-1}}\hat{S}^{(g_1,h_1)}(\Phi_1(\hat{T}^{(g,h)}(x,k))\\ = \hat{S}^{(g^{-1},h^{-1})}\big(\Phi_1(\hat{T}^{(g,h)}(x,k))\big),
\end{multline*}
showing that the definition of $\hat{\Phi}(x,k)$ does not depend on which valid choice we make of $(g,h)$.

Analogous reasoning shows that $\hat{\Phi}$ is equivariant between the two $(G\times H)$-actions.

Clearly $\hat{\Phi}|\hat{U} = \Phi_1$, and for subsets of $\hat{U}$ this amplifies the measure $\hat{\mu}$ by the fixed constant $\comp(\Phi_1)$.  Since $\hat{\Phi}$ is equivariant and both of the systems $(\hat{X},\hat{\mu},\hat{T})$ and $(\hat{Y},\hat{\nu},\hat{S})$ are ergodic, this fact extends to the whole of $\hat{\Phi}$.  This shows that $\hat{\Phi}$ has the desired properties.
\end{proof}

Behind Lemma~\ref{lem:inf-meas-iso} lies a more conceptual fact: $(\hat{X},\hat{\mu},\hat{T})$ and $(\hat{Y},\hat{\nu},\hat{S})$ can be identified with the measure coupling of $G$ and $H$ that arises from the given stable orbit equivalence, as in the proof of~\cite[Theorem 3.3]{Furman99}.  So far in this section we have not assumed that $\Phi$ is a SOE$_\infty$ or SSOE$_1$.  However, if $\Phi$ is a SSOE$_1$, then by definition we may choose the extended cocycles $\hat{\a}$ and $\hat{\b}$ to be integrable.  We therefore obtain the following integrable analog of~\cite[Theorem 3.3]{Furman99}.  This corollary is certainly already known to experts, but we record it explicitly for later reference.

\begin{cor}\label{cor:int-Furman}
If there exists a SSOE$_1$ from a $G$-system to an $H$-system, then $G$ and $H$ are integrably measure equivalent. \qed
\end{cor}

\section{Finite-index subgroups}\label{sec:finite-ind}

Given an SOE between ergodic actions of two groups, and also a finite-index subgroup of each group, one can construct a new SOE between ergodic actions of those subgroups.  The construction is explained in this section in case the subgroups are normal.  Similar arguments can be carried out without the assumption of normality, but extra technicalities arise which we do not address here.  Many of the results we need can be found in~\cite[Sections 2 and 3]{Furman99}, up to the translation between SOE and measure equivalence: see, for instance,~\cite[Example 2.9]{Furman99}.

If $(X,\mu,T)$ is a $G$ system and $G_1 \leq G$ is a subgroup, then $T^{|G_1}$ denotes the restriction of the action to $G_1$.

Our first tool is the following simple lemma.

\begin{lem}\label{lem:erg-cpts}
Let $(X,\mu,T)$ be an ergodic $G$-system, and let $G_1 \unlhd G$ have finite index.  Then there is a finite measurable partition $\cal{P}$ of $X$ into sets of equal measure such that the ergodic components of the system $(X,\mu,T^{|G_1})$ are obtained by conditioning $\mu$ on the cells of $\cal{P}$.
\end{lem}

\begin{proof}
Let $g_1G_1,\dots,g_kG_1$ be the distinct left cosets of $G_1$ in $G$. Let $\A$ be the $\s$-algebra of $T^{|G_1}$-invariant sets.  It is a factor of the whole $G$-action, because $G_1$ is normal in $G$.  If $A \in \A$ has positive measure, then $G_1$-invariance implies that
\begin{equation}\label{eq:union}
\bigcup_{g\in G}T^gA = \bigcup_{i=1}^k T^{g_i}A.
\end{equation}
This set is invariant for the whole $G$-action and has positive measure, so that measure must equal $1$ by ergodicity.  Therefore $\mu(A) \geq 1/k$.

So all members of $\A$ either have measure zero or have measure at least $1/k$, and so $\A$ is atomic modulo negligible sets.  Letting $\cal{P}$ be a set of atoms for $\A$ modulo negligible sets, we obtain that
\begin{itemize}
 \item[(i)] the action $T^{|G_1}$ is ergodic inside each cell of $\cal{P}$, and
\item[(ii)] the action $T$ permutes the cells of $\cal{P}$, and must do so transitively because any union as in~(\ref{eq:union}) has full measure in $X$.
\end{itemize}
Conclusion (i) implies that the ergodic components of $(X,\mu,T^{|G_1})$ are obtained by conditioning on the cells of $\cal{P}$, and conclusion (ii) implies that all those cells have the same measure.
\end{proof}

\begin{lem}\label{lem:ent-and-index}
In the setting of the previous lemma, if $G$ is amenable and $P \in \cal{P}$ then
\[\rmh(\mu_{|P},T^{|G_1}) = [G:G_1]\cdot \rmh(\mu,T).\]
\end{lem}

\begin{proof}
A standard calculation from the definition of KS entropy gives
\begin{equation}\label{eq:ent1}
\rmh(\mu,T^{|G_1}) = [G:G_1]\cdot \rmh(\mu,T).
\end{equation}
On the other hand, since all cells of $\cal{P}$ have equal measure, the affinity of the entropy function gives
\begin{equation}\label{eq:ent2}
\rmh(\mu,T^{|G_1}) = \frac{1}{|\cal{P}|}\sum_{P \in \cal{P}}\rmh(\mu_{|P},T^{|G_1}).
\end{equation}
Lastly, all of the systems $(\mu_{|P},T^{|G_1})$ for $P \in \cal{P}$ are conjugate-isomorphic.  Indeed, if $P,P' \in \cal{P}$, then we may choose $g \in G$ such that $T^gP = P'
$, and now the transformation $\Psi:= T^g$ sends the measure $\mu_{|P}$ to the measure $\mu_{|P'}$ and satisfies
\[\Psi\circ T^{g_1} = T^{\phi(g_1)}\circ \Psi \quad \hbox{for all}\ g_1 \in G_1,\]
where $\phi \in \rm{Aut}(G_1)$ is conjugation by $g$.

Therefore all the summands on the right-hand side of~(\ref{eq:ent2}) are equal, and so by~(\ref{eq:ent1}) they must all be equal to $[G:G_1]\cdot \rmh(\mu,T)$.
\end{proof}

Now let $(X,\mu,T)$ and $(Y,\nu,S)$ be free ergodic $G$- and $H$-systems respectively. Let $\Phi:(X,\mu,T)\rightarrowtail (Y,\nu,S)$ be a SSOE$_1$ or SOE$_\infty$, and let $U:= \dom \Phi$ and $V:=\img \Phi$.  Recall that we always assume $G$ and $H$ are infinite, and let $G_1 \unlhd G$ and $H_1 \unlhd H$ be normal subgroups of finite index.

In this situation, we will construct an SSOE$_1$ or SOE$_\infty$ between a free ergodic $G_1$-system and a free ergodic $H_1$-system so that the new entropies and new compression are related to the old values in the following simple way.

\begin{prop}\label{prop:pass-to-subgps}
There are a free ergodic $G_1$-system $(X_1,\mu_1,T_1)$ and a free ergodic $H_1$-system $(Y_1,\nu_1,S_1)$ such that
\begin{equation}\label{eq:old-ent-new-ent}
\rmh(\mu_1,T_1) = [G:G_1]\cdot \rmh(\mu,T) \quad \hbox{and} \quad \rmh(\nu_1,S_1) = [H:H_1]\cdot \rmh(\nu,S),
\end{equation}
and an SSOE$_1$ (resp. SOE$_\infty$)
\[\Phi_1:(X_1,\mu_1,T_1) \rightarrowtail (Y_1,\nu_1,S_1)\]
such that
\begin{equation}\label{eq:old-comp-new-comp}
\comp(\Phi_1) = \frac{[H:H_1]}{[G:G_1]}\comp(\Phi).
\end{equation}
\end{prop}

This proposition enables one to deduce Theorem A for a pair of groups if it is known for a pair of finite-index subgroups.  This is important in the case of virtually Euclidean groups, which are treated in the next section.

The key to Proposition~\ref{prop:pass-to-subgps} is the infinite-measure-preserving $(G\times H)$-system of Lemma~\ref{lem:inf-meas-iso}.  Let $\hat{\a}:G\times X\to H$ and $\hat{\b}:H\times Y\to G$ be extensions of $\a$ and $\b$ as given by Lemma~\ref{prop:cocyc-extend}, and let $(\hat{X},\hat{\mu},\hat{T})$ and $(\hat{Y},\hat{\nu},\hat{S})$ be the isomorphic $(G\times H)$-systems that appear in Lemma~\ref{lem:inf-meas-iso}.

\begin{proof}
\emph{Step 1.}\quad We first construct a new $G$-system and a new $H$-system.  In a later step we will restrict these to $G_1$ and $H_1$ and then obtain $(X_1,\mu_1,T_1)$ and $(Y_1,\nu_1,S_1)$ as ergodic components of those restrictions.

Let $\ol{G} := G/G_1$ and $\ol{H} := H/H_1$ be the finite quotient groups.  For elements $g \in G$ and $h \in H$ let $\ol{g} \in \ol{G}$ and $\ol{h} \in \ol{H}$ be their respective images.

Now let $X_1 := X\times \ol{H}$ and $Y_1 := Y\times \ol{G}$.  Let $\t{\mu}$ on $X_1$ be the product of $\mu$ and the Haar measure on $\ol{H}$, and define $\t{\nu}$ on $Y_1$ similarly.  Let $\pi_X:X_1\to X$ and $\pi_Y:Y_1\to Y$ be the coordinate projections, so these are $[H:H_1]$-to-$1$ and $[G:G_1]$-to-$1$, respectively.

Define a $G$-action $\t{T}$ on $(X_1,\t{\mu})$ by
\[\t{T}^g(x,\ol{h}) := (T^gx,\ol{\hat{\a}(g,x)}\cdot \ol{h}),\]
and similarly define an $H$-action $\t{S}$ on $(Y_1,\t{\nu})$ by
\[\t{S}^h(y,\ol{g}) := (S^hy,\ol{\hat{\b}(h,y)}\cdot \ol{g}).\]
Then $\pi_X$ intertwines $\t{T}$ with $T$ and $\pi_Y$ intertwines $\t{S}$ with $S$.

\vspace{7pt}

\emph{Step 2.}\quad Now let $T_1 := \t{T}^{|G_1}$ and $S_1 := \t{S}^{|H_1}$.  Doing so gives a free $G_1$-system $(X_1,\t{\mu},T_1)$ and a free $H_1$-system $(Y_1,\t{\nu},S_1)$.  These systems need not be ergodic.

We set the issue of ergodicity aside for now, and next construct an SOE between these systems.  Let
\[U_1 := U\times \{e_{\ol{H}}\} \quad \hbox{and} \quad V_1 := V\times \{e_{\ol{G}}\},\]
and observe that
\begin{equation}\label{eq:newmeass}
\t{\mu}(U_1) = \frac{1}{[H:H_1]}\mu(U) \quad \hbox{and} \quad \t{\nu}(V_1) = \frac{1}{[G:G_1]}\nu(V).
\end{equation}

Define $\t{\Phi}:U_1\to V_1$ by $\t{\Phi}(x,e_{\ol{H}}) := (\Phi(x),e_{\ol{G}})$.  This is an SOE from $(X_1,\t{\mu},T_1)$ to $(Y_1,\t{\nu},S_1)$.  To see this, suppose that $(x,e_{\ol{H}}) \in U_1$ and $g \in G_1$ are such that $T_1^g(x,e_{\ol{H}}) = (T^g x,\ol{\hat{\a}(g,x)}) \in U_1$.  Since $x,T^gx \in U$, we have $\hat{\a}(g,x) = \a(g,x)$.  Now the following both hold:
\begin{itemize}
 \item[(i)] The points $x$ and $T^gx$ lie in the same class of $\R_T\cap (U\times U)$, and hence their $\Phi$-images lie in the same class of $\R_S\cap (V\times V)$, because $\Phi$ is an SOE.
\item[(ii)] Since $T_1^g(x,e_{\ol{H}}) \in U_1$, we must have $\ol{\a(g,x)} = e_{\ol{H}}$, and hence $\a(g,x) \in H_1$.  Therefore the points $\Phi(x)$ and $\Phi(T^gx) = S^{\a(g,x)}\Phi(x)$ actually lie in the same $H_1$-orbit, not just the same $H$-orbit.
\end{itemize}
These conclusions show that $\t{\Phi}$ maps the classes of $\R_{T_1}\cap (U_1\times U_1)$ into classes of $\R_{S_1}\cap (V_1\times V_1)$.  By the symmetry of the construction, the same holds in reverse, and so $\t{\Phi}$ is an SOE as required.

Observe that the calculation~(\ref{eq:newmeass}) gives
\[\comp(\t{\Phi}) = \frac{\t{\nu}(V_1)}{\t{\mu}(U_1)} = \frac{[H:H_1]}{[G:G_1]}\comp(\Phi),\]
where we use the measures $\t{\mu}$ and $\t{\nu}$ on our two new systems.

\vspace{7pt}

\emph{Step 3.}\quad We will now replace $\t{\mu}$ and $\t{\nu}$ with ergodic measures so as to preserve the properties obtained above.

This relies on the following observation.  Let us identify $G$ and $G_1$ with the corresponding subgroups in the first coordinate of $G\times H$, and similarly for $H$ and $H_1$. Then the space $(X_1,\t{\mu})$ may be identified with a fundamental domain for the action $\hat{T}^{|H_1}$ on the infinite measure space $(\hat{X},\hat{\mu})$, and so $(X_1,\t{\mu},T_1)$ may be identified with the factor of $(\hat{X},\hat{\mu},\hat{T}^{|G_1})$ consisting of $\hat{T}^{H_1}$-invariant sets.  Similarly, $(Y_1,\t{\nu},S_1)$ may be identified with the $\hat{S}^{G_1}$-invariant factor of $(\hat{Y},\hat{\nu},\hat{S}^{|H_1})$.

As a result, there is a measure-preserving $(G\times H)$-action on $(X_1,\t{\mu})$ given by the quotient of the full $(G\times H)$-system $(\hat{X},\hat{\mu},\hat{T})$, and $T_1$ is the restriction of that $(G\times H)$-action to $G_1$.  This $(G\times H)$-action on $(X_1,\t{\mu})$ is ergodic, because the infinite-measure-preserving system above it is ergodic.  On the other hand, $G_1 \cong G_1\times \{e_H\}$ is normal in $G\times H$, because $G_1$ is normal in $G$. We may therefore apply Lemma~\ref{lem:erg-cpts} to the inclusion of $T_1$ into this larger ergodic $(G\times H)$-action.  It tells us that the $G_1$-system $(X_1,\t{\mu},T_1)$ has some finite number, say $n$, of ergodic components, and each of them is obtained by conditioning on an invariant set of measure $1/n$.  Let $\cal{P}$ be the partition of $X_1$ consisting of these components. An analogous argument gives a finite partition $\cal{Q}$ of $(Y_1,\t{\nu})$ into equal-measure ergodic components for $S_1$; let $m$ be the number of these.

Crucially, we can now show that $n=m$.  Since $(X_1,\t{\mu})$ is the quotient of $(\hat{X},\hat{\mu})$ by the action $\hat{T}^{|H_1}$, we may identify $\cal{P}$ with the partition of $(\hat{X},\hat{\mu})$ into ergodic components for the combined action $\hat{T}^{|G_1\times H_1}$.  Similarly, $\cal{Q}$ may be identified with the partition of $(\hat{Y},\hat{\nu})$ into ergodic components for the combined action $\hat{S}^{|G_1\times H_1}$.  But those two actions are isomorphic up to a constant change of measure, by Lemma~\ref{lem:inf-meas-iso}, and so they have the same numbers of ergodic components.

To finish our construction, choose one of the ergodic components $P \in \cal{P}$ for which $\t{\mu}(P\cap U_1) > 0$.  The restriction of $\cal{P}$ to $U_1$ gives the ergodic decomposition of $\R_{T_1}\cap (U_1\times U_1)$ up to negligible sets, and $\t{\Phi}$ carries that restriction to the ergodic decomposition of $\R_{S_1}\cap (V_1\times V_1)$.  Therefore $\t{\Phi}$ identifies $P\cap U_1$ with $Q\cap V_1$ for a unique cell $Q\in \cal{Q}$.  Now let $\mu_1:= \t{\mu}_{|P}$ and $\nu_1 := \t{\nu}_{|Q}$.  Then the restriction $\Phi_1 := \t{\Phi}|P\cap U_1$ defines a SOE from $(X_1,\mu_1,T_1)$ to $(Y_1,\nu_1,T_1)$, and these are a free ergodic $G_1$-system and a free ergodic $H_1$-system respectively.

For these systems, we may calculate the entropy using Lemma~\ref{lem:ent-and-index}.  On the other hand, we observe that $\comp(\Phi_1)$ just equals $\comp(\t{\Phi})$ because $n=m$.  Note that, although $\Phi_1$ is simply a restriction of $\t{\Phi}$ to a subset, the equality $n=m$ is needed for this second calculation because the measures have also been changed: from $\t{\mu}$ and $\t{\nu}$ to their restrictions $\mu_1$ and $\nu_1$.

The partial cocycles associated to $\Phi_1$ and $\Phi_1^{-1}$ are simply restrictions of those associated to $\Phi$ and $\Phi^{-1}$. Also, the new measures $\mu_1$ and $\nu_1$ have bounded Radon--Nikodym derivatives with respect to $\t{\mu}$ and $\t{\nu}$ respectively.  Therefore $\Phi_1$ is an SSOE$_1$ (resp. SOE$_\infty$) if $\Phi$ has this property.
\end{proof}

\section{Virtually Euclidean groups}\label{sec:Ab}

This section proves Theorem B, which concerns a SOE$_\infty$ or SSOE$_1$ between actions of Euclidean lattices.  From this we deduce Theorem A in case $G$ and $H$ are both virtually Euclidean: that is, they contain finite-index subgroups isomorphic to Euclidean lattices.  By intersecting finitely many conjugates, one may assume that those subgroups are normal.  Let $e_1$, \dots, $e_d$ be the standard basis of $\bbZ^d$, and let $|\cdot|$ be the corresponding $\ell^1$-norm on $\bbZ^d$.

The special case of Euclidean lattices is important for two reasons.  Firstly, we will make contact with the older notion of Kakutani equivalence for actions of Euclidean lattices, which has been studied much more thoroughly than SSOE$_1$.  I do not know whether these notions are actually equivalent.  Secondly, our approach to Theorem A in the remainder of the paper needs the assumption that $G$ and $H$ have super-linear growth, so it does not cover virtually cyclic groups.  We therefore need the results of the present section to prove Theorem A in that case.  For groups containing a finite-index copy of $\bbZ^d$ with $d \geq 2$, we end up with two proofs of Theorem A, one in the present section and the other from the remainder of the paper.

We start with Theorem B, which applies to Euclidean lattices themselves, and then prove Theorem A for virtually Euclidean groups using Theorem B and Proposition~\ref{prop:pass-to-subgps}.

In the Euclidean case, Theorem B connects SOE$_\infty$ and SSOE$_1$ with the generalization of Kakutani equivalence to $\bbZ^d$-actions developed in~\cite{Kat77,delJRud84,Has92}.  We use the definition of this property from~\cite[Definition 3]{delJRud84}:

\begin{dfn}\label{dfn:Kak}
	Let $M$ be a real $(d\times d)$-matrix.  Two $\bbZ^d$-systems $(X,\mu,T)$ and $(Y,\nu,S)$ are \textbf{$M$-Kakutani equivalent} if there is an SOE $\Phi:(X,\mu,T)\rightarrowtail (Y,\nu,S)$ with the following properties:
	\begin{itemize}
		\item[(i)] $\dom \Phi = X$, and
		\item[(ii)] if $\a:\bbZ^d\times X\to \bbZ^d$ is the cocycle describing $\Phi$, then for any $\eps > 0$ there are $N_\eps \in \bbN$ and $A_\eps \subseteq X$ with $\mu(A_\eps) > 1-\eps$ such that, if $v \in \bbZ^d$ has $|v| \geq N$, and $x \in A_\eps \cap T^{-v}A_\eps$, then
		\[|\a(v,x) - Mv| \leq \eps|v|.\]
	\end{itemize}
\end{dfn}

In their paper, del Junco and Rudolph refer to $\Phi$ as an `orbit injection', rather than a `SOE', and say that it `maps distinct orbits into distinct orbits'.  They also make the explicit assumption that $M$ is invertible with $|\det M|\geq 1$.  However, the paragraph immediately following the proof of their Proposition 3 makes it clear that this is what we call a SOE, and that the other parts of Definition~\ref{dfn:Kak} actually require that $|\det M|\geq 1$.

It is helpful to know that part (ii) of Definition~\ref{dfn:Kak} can be replaced by the following apparently weaker condition:
\begin{itemize}
\item[\emph{(ii)$'$}]	\emph{For any $\eps > 0$ there are $N_\eps \in \bbN$ and $A_\eps \subseteq X$ with $\mu(A_\eps) > 1 - \eps$ such that, if $1 \leq i \leq d$, $n \geq N_\eps$, and $x \in A_\eps \cap T^{-ne_i}A_\eps$, then
	\[|\a(ne_i,x) - nMe_i| < \eps n.\]
}
\end{itemize}
The condition that this holds for some basis in $\bbZ^d$ is Condition 1 on p93 of~\cite{delJRud84}.  The fact that it implies $M$-Kakutani equivalence is their Proposition 7.

The first assertion of Theorem B reduces our work to the case of SSOE$_1$.  We isolate it as the following lemma.

\begin{lem}
If $(X,\mu,T)$ is a $\bbZ^d$-system, $(Y,\nu,S)$ is a $\bbZ^D$-system, and they are SOE$_\infty$, then the partial cocycles $\a$ and $\b$ which describe this SOE have extensions to full cocycles $\bbZ^d\times X\to \bbZ^D$ and $\bbZ^D\times Y\to \bbZ^d$ which are still bounded.  In particular, the systems are SSOE$_1$.
\end{lem}

\begin{proof}
It suffices to show that any bounded $\bbZ^D$-valued partial cocycle over $(X,\mu,T)$ can be extended to a bounded cocycle $\bbZ^d\times X \to \bbZ^D$.  Arguing coordinate-wise it suffices to prove this when $D=1$.  Thus, let $(\a,U)$ be a $\bbZ$-valued partial cocycle over $(X,\mu,T)$, and assume that $|\a(v,x)| \leq C|v|$ for $\mu$-a.e. $x \in U\cap T^{-v}U$, for all $v \in \bbZ^d$.

For each $x \in X$ let
\[D_x := \{v \in \bbZ^d:\,T^vx \in U\},\]
the $U$-return set of $x$.  Since $(X,\mu,T)$ is ergodic and $\mu(U) > 0$, this $D_x$ is nonempty for almost every $x$.  By removing a negligible set, we may assume this holds for strictly every $x$.

Now consider $x \in U$, so $0 \in D_x$.  Then the assumed boundedness of $\a$ is equivalent to the assertion that the map
\[D_x\to \bbZ:v\mapsto \a_x(v)\]
is $C$-Lipschitz for the restriction of $|\cdot|$ to $D_x$.  We may therefore apply a standard construction to extend it to a $C$-Lipschitz map from the whole of $\bbZ^d$ to $\bbR$, and then apply some rounding to produce a $\bbZ$-valued function.  To be specific, for $u \in \bbZ^d$, let us define
\[\s^0_x(u) := \big\lfloor \min\{\a_x(v) + C\|u-v\|:\ v \in D_x\}\big\rfloor,\]
where $\lfloor\cdot\rfloor$ is the integer-part function.  This is $(C+2)$-Lipschitz, where the extra `$2$' allows for the rounding.  It extends $\a_x$, and it satisfies the following slightly extended cocycle identity:
\begin{equation}\label{eq:ext-cocyc}
\s^0_x(u+w) = \s^0_x(u) + \s^0_{T^ux}(w) = \a_x(u) + \s^0_{T^ux}(w) \quad \hbox{whenever}\ x,T^ux \in U.
\end{equation}

Finally, the cocycle equation tells us how to extend $\s^0$ further to a function on the whole of $\bbZ^d \times X$. For each $x\in X$, choose some $v \in D_x$, and let
\begin{equation}\label{eq:sigma-def}
\s_x(u) := \s^0_{T^vx}(u-v) - \s^0_{T^vx}(-v).
\end{equation}
A re-arrangement using equation~(\ref{eq:ext-cocyc}) shows that this right-hand side does not depend on $v$, so $\s_x(u)$ is well-defined.  If $x \in U$ then we may use the choice $v = 0$, which shows that $\s$ does indeed extend $\s^0$.  The new function $\s_x$ is still $(C+2)$-Lipschitz on $\bbZ^d$ for each $x$ because $\s^0_{T^vx}$ has that property.

It remains to verify the cocycle identity for $\s$.  Suppose that $x \in X$ and $u,w \in \bbZ^d$, and choose $v \in D_x$.  It follows that $v - w \in D_{T^wx}$.  Therefore, using these two points in the right-hand side of~(\ref{eq:sigma-def}), we obtain
\begin{align*}
	\s_x(u+w) &= \s^0_{T^vx}(u+w-v) - \s^0_{T^vx}(-v)\\
	&= \s^0_{T^{v-w}(T^wx)}(u-(v-w)) - \s^0_{T^{v-w}(T^wx)}(-(v-w))\\ &\quad + \s^0_{T^vx}(w-v) - \s^0_{T^vx}(-v)\\
	&= \s_{T^wx}(u) + \s_x(w),
\end{align*}
as required.
\end{proof}

\begin{proof}[Proof of Theorem B]
By the preceding lemma, it suffices to assume that
\[\Phi:(X,\mu,T)\rightarrowtail (Y,\nu,S)\]
is an SSOE$_1$.  By considering $\Phi^{-1}$ instead if necessary, we may assume that $\comp(\Phi) \leq 1$.

This SSOE$_1$ between the systems implies an integrable measure equivalence between the two groups, by Corollary~\ref{cor:int-Furman}.  As shown by Lewis Bowen in~\cite[Theorem B.2]{Aus--nilpIME}, this requires that they have the same growth, and hence $D = d$.

Let $U := \dom \Phi$ and $V := \img \Phi$, and let $\a:\bbZ^d\times Y\to \bbZ^d$ be an integrable cocycle such that $(\a,U)$ describes $\Phi$.

Since $\comp(\Phi) \leq 1$, we have $\nu(V) \leq \mu(U)$.  Choose a measurable subset $W\subseteq Y$ such that $W\supseteq V$ and $\nu(V)/\nu(W) = \mu(U)$. By~\cite[Proposition 2.7]{Furman99}, $\Phi$ has an extension to an isomorphism $\t{\Phi}$ between the relations $\R_T$ and $\R_S\cap (W\times W)$: that is, $\t{\Phi}$ is a SOE which extends $\Phi$, whose domain is the whole of $X$, and whose image is $W$.  It has the same compression as $\Phi$.  Since $\dom\t{\Phi} = X$, it is described by a cocycle $\s:\bbZ^d\times X\to \bbZ^d$ such that $\s_{|U} = \a_{|U}$ and such that $\s_x$ is an injection for a.e. $x$.  It does not follow that $\s$ is integrable, but since $\s_{|U} = \a_{|U}$, the second part of Proposition~\ref{prop:cocyc-extend} promises that $\s$ is cohomologous to $\a$, say
\[\s(v,x) = \a(v,x) + \g(T^vx) - \g(x)\]
for some $\g:X\to \bbZ^d$.

Next, since $\a$ is integrable, the cocycle equation and the pointwise ergodic theorem give that
\begin{equation}\label{eq:PET}
\frac{\a(ne_i,x)}{n} = \frac{1}{n}\sum_{j=0}^{n-1}\a(e_i,T^{je_i}x) \to v_i := \int \a(e_i,x)\,\mu(dx) \quad \hbox{as}\ n\to\infty
\end{equation}
for $\mu$-almost every $x$ and for $i=1,2,\dots,d$. Let $M$ be the $(d\times d)$-matrix whose columns are the vectors $v_i$.

We now show that the SOE $\t{\Phi}$ is an $M$-Kakutani equivalence for this $M$. We have guaranteed condition (i) by construction, and we finish the proof by showing condition (ii)$'$ instead of (ii).  Given $\eps > 0$, choose $r_\eps < \infty$ so large that the set
\[B_\eps := \{x:\ |\g(x)| \leq r_\eps\}\]
has $\mu(B_\eps) > 1-\eps/2$. Now choose $N_\eps$ so large that
\[r_\eps < \eps N_\eps/2\]
and so that the set
\[C_\eps := \big\{x:\ |\a(ne_i,x) - nv_i| < \eps n/2 \ \ \forall n\geq N_\eps\ \forall i=1,2,\dots,d\big\}\]
has $\mu(C_\eps) > 1-\eps/2$; this is possible because of~(\ref{eq:PET}).  Finally, let $A_\eps := B_\eps \cap C_\eps$.  Then $\mu(A_\eps) > 1-\eps$, and for any $n \geq N_\eps$ and $x \in A_\eps \cap T^{-ne_i}A_\eps$ we obtain
\[|\s(ne_i,x) - nMe_i| \leq |\a(ne_i,x) - nv_i| + |\g(x)| + |\g(T^{ne_i}x)| < \eps n/2 + 2r_\eps < \eps n.\]
\end{proof}

\begin{cor}
The conclusion of Theorem A holds if $G$ and $H$ are virtually Euclidean.
\end{cor}

\begin{proof}
If $G$ and $H$ are strictly Euclidean, then Theorem B reduces this to the corresponding result for Kakutani equivalence.  By the explanation which follows the proof of Proposition 3 in~\cite{delJRud84}, the matrix $M$ constructed in the proof of Theorem B must satisfy
\[\comp(\Phi) = \comp(\t{\Phi}) = \frac{1}{|\det M|}.\]
Now the desired result follows from the equation
\[\rmh(\nu,S) = \frac{\rmh(\mu,T)}{|\det M|},\]
which is recalled on the last page of~\cite{delJRud84} (beware that this equation also appears at the bottom of p91 of their paper, but written incorrectly).  Del Junco and Rudolph attribute this equation to an unpublished work of Nadler, but the special case of $\rm{id}$-Kakutani equivalence is included as~\cite[Corollary 3]{Has92}, and the general case is proved in the same way.

Now suppose that $G_1 \unlhd G$ and $H_1 \unlhd H$ are finite-index subgroups isomorphic to Euclidean lattices.  Let $(X_1,\mu_1,T_1)$, $(Y_1,\nu_1,S_1)$ and $\Phi_1$ be the systems and SOE given by Proposition~\ref{prop:pass-to-subgps}.  Then the special case of Euclidean groups gives that
\[\rmh(\nu_1,S_1) = \comp(\Phi_1)\rmh(\mu_1,T_1),\]
and now the equations~(\ref{eq:old-ent-new-ent}) and~(\ref{eq:old-comp-new-comp}) turn this into the desired conclusion.
\end{proof}

\begin{rmk}
	Beyond Kakutani equivalence for $\bbZ^d$-actions, Kammeyer and Rudolph have developed a very abstract notion of `restricted orbit equivalences' between actions of discrete amenable groups: see~\cite{KamRud97,KamRud02}.  I do now know whether OE$_1$ or SSOE$_1$ are examples of restricted orbit equivalences, but if so then their machinery would have several consequences in our setting, such as an analog of Ornstein theory. \fin
\end{rmk}

\begin{ques}
	Is it true that SOE$_\infty$ implies SSOE$_1$ between actions of other finitely generated amenable groups? Does Kakutani equivalence imply either?
	\fin
\end{ques}

\section{Proof of the entropy formula using derandomization}\label{sec:AfromC}

This section returns to the setting of general amenable-group actions.  It derives Theorem A from Theorem D.  The more difficult work of proving Theorems C and D occupies the rest of the paper after this.

Theorem D leads to Theorem A via the following.

\begin{prop}\label{prop:rel-ent-pres}
Let
\begin{center}
$\phantom{i}$\xymatrix{
(X,\mu,T) \ar@{>->}^\Phi[r] \ar_\pi[d] & (Y,\nu,S) \ar^\xi[d]\\
(X',\mu',T') \ar@{>->}_{\Phi'}[r] & (Y',\nu',S')
}
\end{center}
be a commutative diagram whose rows are SOEs, whose left column is a factor map of free $G$-systems, and whose right column is a factor map of free $H$-systems.  Then the relative entropies over those factor maps satisfy
\[\mu(\dom\Phi)^{-1}\rmh(\mu,T\,|\,\pi) = \nu(\img \Phi)^{-1}\rmh(\nu,S\,|\,\xi).\]
\end{prop}

This result may already be known, but I have not found a suitable reference in the literature.  It may be a consequence of Danilenko's quite abstract results in~\cite[Section 2]{Dan01}, but it seems worth including a more classical proof.  A simple approach, suggested to me by Lewis Bowen, is based on the following lemma.

\begin{lem}\label{lem:relCFW}
Let $\pi:(X,\mu,T)\to (X',\mu',T')$ be a factor map of free $G$-systems.  Then there is a commutative diagram
\begin{center}
$\phantom{i}$\xymatrix{
(X,\mu,R) \ar^{\id_X}[r] \ar_\pi[d] & (X,\mu,T) \ar^\pi[d]\\
(X',\mu',R') \ar_{\id_{X'}}[r] & (X',\mu',T')
}
\end{center}
in which $R$ and $R'$ are single transformations and $\id_X$ and $\id_{X'}$ are OEs (equivalently, $R$ and $T$ have the same orbits and $R'$ and $T'$ have the same orbits).
\end{lem}

\begin{proof}
By the main result of~\cite{ConFelWei81}, there is a single $\mu$-preserving transformation $R'$ on $X'$ which has the same orbits as the action $T'$.  Since $(X',\mu',T')$ is free, this implies the existence of a unique cocycle $\a:X'\to G$ such that
\[R'x' = (T')^{\a(x')}x' \quad \hbox{for}\ x' \in X'.\]
The proof is completed by defining
\[Rx := T^{\a(\pi(x))}x \quad \hbox{for}\ x \in X.\]
\end{proof}

Lemma~\ref{lem:relCFW} enables us to convert $G$- and $H$-actions into $\bbZ$-actions, for which stable orbit equivalence is easier to understand.  For $\bbZ$-actions, stable orbit equivalence is simply an orbit equivalence between induced transformations, whose entropy is computed by Abramov's formula.

\begin{proof}[Proof of Proposition~\ref{prop:rel-ent-pres}]
Let $U:= \dom\Phi$, $U' := \dom\Phi'$, $V := \img \Phi$ and $V' := \img\Phi'$.  Our assumptions include that $U = \pi^{-1}U'$ and $V = \xi^{-1}V'$.

First we invoke Lemma~\ref{lem:relCFW} on the left-hand side of the diagram in the statement of Proposition~\ref{prop:rel-ent-pres}.  This produces the larger diagram
\begin{center}
$\phantom{i}$\xymatrix{
(X,\mu,R) \ar^{\id_X}[r] \ar_\pi[d] & (X,\mu,T) \ar@{>->}^\Phi[r] \ar_\pi[d] & (Y,\nu,S) \ar^\xi[d]\\
(X',\mu',T') \ar_{\id_{X'}}[r] & (X',\mu',T') \ar@{>->}_{\Phi'}[r] & (Y',\nu',S').
}
\end{center}
By~\cite[Theorem 2.6]{RudWei00}, the left-hand square above gives the equality
\begin{eqnarray}\label{eq:RW}
\rmh(\mu,R\,|\,\pi) = \rmh(\mu,T\,|\,\pi).
\end{eqnarray}

Now composing the rows of this diagram, it collapses to
\begin{center}
$\phantom{i}$\xymatrix{
(X,\mu,R) \ar@{>->}^\Phi[r] \ar_\pi[d] & (Y,\nu,S) \ar^\xi[d]\\
(X',\mu',R') \ar@{>->}_{\Phi'}[r] & (Y',\nu',S').
}
\end{center}
In view of~(\ref{eq:RW}), it suffices to show that this diagram implies the equality
\[\mu(U)^{-1}\rmh(\mu,R\,|\,\pi) = \nu(V)^{-1}\rmh(\nu,S\,|\,\xi):\]
that is, we have reduced the desired proposition to the case $G = \bbZ$.

Applying Lemma~\ref{lem:relCFW} in the same way on the right-hand side of the diagram, we may reduce to the case in which $G = H = \bbZ$, and so $T$ and $S$ may be regarded as single transformations.  However, in this case $\Phi$ (resp. $\Phi'$) is an OE between the induced transformations $T_U$ and $S_V$ (resp. $T'_{U'}$ and $S'_{V'}$), and so another appeal to~\cite[Theorem 2.6]{RudWei00} gives
\[\rmh\big(\mu_{|U},T_U\,\big|\,\pi|U\big) = \rmh\big(\nu_{|V},S_V\,\big|\,\xi|V\big).\]
Finally, Abramov's formula for the entropy of induced transformations~\cite{Abr59} and the Abramov-Rokhlin formula for the entropy of an extension~\cite{AbrRok62} give
\begin{eqnarray*}
\rmh\big(\mu_{|U},T_U\,\big|\,\pi|U\big) &=& \rmh\big(\mu_{|U},T_U\big) - \rmh\big(\mu'_{|U'},T'_{U'}\big)\\
&=& \mu(U)^{-1}\big(\rmh(\mu,T) - \rmh(\mu',T')\big)\\
&=& \mu(U)^{-1}\rmh(\mu,T\,|\,\pi),
\end{eqnarray*}
and similarly for $\rmh\big(\nu_{|V},S_V\,\big|\,\xi|V\big)$.
\end{proof}

\begin{proof}[Completed proof of Theorem A, given Theorem D]
First suppose that either $G$ or $H$ has linear growth. Lewis Bowen has shown in~\cite[Theorem B.2]{Aus--nilpIME} that growth type is an invariant of integrable measure equivalence, so this implies that they both have linear growth, and hence they are both virtually $\bbZ$.  So in this case the result follows from Section~\ref{sec:Ab}.

Now suppose that both groups have super-linear growth.  Let $\Phi:(X,\mu,T)\rightarrowtail (Y,\nu,S)$ be either a SSOE$_1$ or a SOE$_\infty$, and let $c := \comp(\Phi)$. In this case Theorem D gives a positive-measure subset $U \subseteq \dom\Phi$ and a diagram of the form
\begin{center}
$\phantom{i}$\xymatrix{
(X,\mu,T) \ar@{>->}^{\Phi|U}[r] \ar_\pi[d] & (Y,\nu,S) \ar^\xi[d]\\
(X',\mu',T') \ar@{>->}_{\Phi'}[r] & (Y',\nu',S'),
}
\end{center}
where $\rmh(\mu',T') < \eps$ and all the systems are free.  It follows that
\[c = \frac{\nu(\Phi(U))}{\mu(U)} = \frac{\nu'(\img\Phi')}{\mu'(\dom\Phi')}.\]
We now combine Proposition~\ref{prop:rel-ent-pres} with Ward and Zhang's generalization of the Abramov--Rokhlin formula to extensions of amenable-group actions~\cite[Theorem 4.4]{WarZha92}.  This gives
\begin{eqnarray*}
\rmh(\nu,S) &\geq& \rmh(\nu,S\,|\,\xi)\\
&=& c\rmh(\mu,T\,|\,\pi)\\
&=& c\big(\rmh(\mu,T) - \rmh(\mu',T')\big)\\
&\geq& c\big(\rmh(\mu,T) - \eps\big).
\end{eqnarray*}
Since $\eps > 0$ was arbitrary, it follows that
\[\rmh(\nu,S) \geq c\rmh(\mu,T),\]
and the reverse inequality holds by symmetry.
\end{proof}

\section{Subrelations, graphings and a new notion of cost}\label{sec:graphings}

Most of the rest of the paper will go towards proving Theorem D.  The next step is to introduce some more kinds of structure that will be used during the proof.

\subsection{Graphings}

Let $(X,\mu,T)$ be a $G$-system and $\R_T \subseteq X\times X$ its orbit equivalence relation.  In this setting, we need some definitions related to graphings and their costs.  Graphings go back to Adams' paper~\cite{Ada90}, and cost to Levitt's work~\cite{Lev95}.  These constructions have since become very important to the study of Borel equivalence relations: see, for instance, Gaboriau's survey~\cite{Gab02}.

In order to study integrable orbit equivalence, we need to work with graphings that are always defined with reference to the given $G$-action, and then with a modified notion of cost that accounts for the lengths of elements of $G$.  We therefore adjust the older definitions in the following way.  A \textbf{$T$-graphing} is a family $\G = (A_g)_{g \in G}$ of measurable subsets of $X$ indexed by $G$ satisfying
\begin{equation}\label{eq:symmetry}
A_{g^{-1}} = T^gA_g \quad \forall g \in G.
\end{equation}
The associated graphing in Levitt's sense is the family of partial maps $T^g|A_g:A_g\to T^gA_g$.

The \textbf{vertex set} of a graphing $\G$ is $\rm{Vert}(\G) := \bigcup_g A_g$, and $\G$ is \textbf{nontrivial} if this set has positive $\mu$-measure.  If $\rm{Vert}(\G) = V$, we may regard $\G$ as placing the structure of a graph on each of the equivalence classes in $\R_T\cap (V\times V)$, where $x$ and $T^gx$ are joined by an edge if $x \in A_g$.  Condition~(\ref{eq:symmetry}) is equivalent to this set of edges being symmetric, so we may regard this graph as undirected.

The \textbf{equivalence relation} generated by a $T$-graphing $\G$ is the smallest Borel equivalence relation which contains $(x,T^gx)$ whenever $g\in G$ and $x \in A_g$.  It is denoted by $\R_\G$.

\begin{dfn}
A $T$-graphing $\G$ is \textbf{orbit-wise connected} if
\[\R_\G\cap (V_0\times V_0) = \R_T\cap (V_0\times V_0)\]
for some $V_0\subseteq \Vert(\G)$ with $\mu(\Vert(\G)\setminus V_0) = 0$.  Equivalently, this asserts that for $\mu$-a.e. $x \in \Vert(\G)$, the edges of $\R_\G$ define a connected graph on the set $T^Gx \cap V_0$.
\end{dfn}

The factor of $(X,\mu,T)$ \textbf{generated} by the graphing $\G = (A_g)_g$ is simply the smallest factor which contains all the sets $A_g$.  We write $\rmh(\mu,T,\G)$ for the KS entropy of this factor.

\subsection{Graphings and partial cocycles}

Now suppose that $(\a,U)$ is an $H$-valued partial cocycle over $(X,\mu,T)$ and that $\G = (A_g)_g$ is a $T$-graphing.  Let $V:= \Vert(\G)$, and assume that $V \subseteq U$.  In view of the relation~(\ref{eq:symmetry}), this implies
\[\hbox{both}\quad A_g \subseteq U \quad\hbox{and}\quad A_{g^{-1}} = T^gA_g \subseteq U \quad \forall g \in G,\]
so in fact $A_g \subseteq U\cap T^{g^{-1}}U$ for every $g$.
We may therefore define the \textbf{restriction of $\a$ to $\G$} to be the restriction of $\a$ to the subset
\[\{(g,x):\ g \in G\ \hbox{and}\ x \in A_g\}.\]
Denote it by $\a|\G$.  If $\G$ is the `na\"\i ve' graphing defined by $A_g := V \cap T^{g^{-1}}V$ for every $g$, then this agrees with our previous definition of $\a_{|V}$. The factor \textbf{generated} by $\a|\G$ is the factor $\A$ generated by all the partial observables $(\a^g|A_g,A_g)$, $g \in G$, and its entropy is
\[\rmh(\mu,T,\a|\G) := \rmh(\mu,T,\A).\]

\begin{lem}\label{lem:cocyc-gen}
If $\G$ is an orbit-wise connected $T$-graphing, and $V:= \Vert(\G)$, then the factor generated by $\a|\G$ contains the factor generated by $(\a_{|V},V)$ up to negligible sets.
\end{lem}

\begin{proof}
Let $\A$ be the factor generated by $\a|\G$.  It contains every $A_g$, so it contains their union $V$, and so it contains all of the intersections $V\cap T^{g^{-1}}V$ (although these need not be equal to $A_g$ for any $g$).

Now fix $g \in G$ and $h \in H$, and consider the subsets
\[V_{g,h} := \{x \in V\cap T^{g^{-1}}V:\ \a(g,x) = h\}.\]
As $g$ and $h$ vary, these generate the $\s$-algebra of $(\a_{|V},V)$.  Since $\G$ is orbit-wise connected, we may remove a negligible set so that a point $x \in V\cap T^{g^{-1}}V$ lies in $V_{g,h}$ if and only if there is a factorization
\[g = g_kg_{k-1}\cdots g_1\]
such that
\[T^{g_{i-1}\cdots g_1}x \in A_{g_i}\quad \forall i=1,\ldots,k\]
and
\[\a(g,x) = \a(g_k,T^{g_{k-1}\cdots g_1}x)\cdots \a(g_1,x) = h.\]
There are only countably many possibilities for the sequence of elements $g_1$, \dots, $g_k \in G$, and similarly for the sequence of elements $\a(g_1,x)$, \dots, $\a(g_k,T^{g_{k-1}\cdots g_1}x)$. Therefore we have expressed $V_{g,h}$ is a countable union of further subsets all of which manifestly lie in the factor generated by $\a|\G$.
\end{proof}

\begin{lem}\label{lem:Shannon-bound2}
Suppose that $(\a,U)$ is a partial cocycle over $(X,\mu,T)$ and that $\G = (A_g)_g$ is a $T$-graphing for which $\Vert(\G) \subseteq U$.  Then
\[\rmh(\mu,T,\a|\G) \leq \sum_{g \in G}\rmH_\mu(A_g) + \sum_{g\in G}\mu(A_g)\rmH_{\mu_{|A_g}}(\a^g).\]
\end{lem}

\begin{proof}
This is a simple application of equation~(\ref{eq:Shannon-bound}):
\begin{align*}
\rmh(\mu,T,\a|\G) &\leq \sum_{g \in G}\rmh(\mu,T,(\a^g|A_g,A_g)) \leq \sum_{g\in G}\rmH_\mu(\a^g;A_g) \nonumber\\
&= \sum_{g \in G}\rmH_\mu(A_g) + \sum_{g \in G}\mu(A_g)\rmH_{\mu_{|A_g}}(\a^g).
\end{align*}
\end{proof}

In combination, the previous two lemmas allow one to control the entropy of the factor generated by $(\a_{|V},V)$ using any choice of orbit-wise connected graphing with vertex set $V$.  A careful choice of that graphing can give a better upper bound than a more na\"\i ve estimate in terms of the partial observables $(\a^g,V\cap T^{g^{-1}}V)$.

The next definition gives our modified notion of cost.

\begin{dfn}\label{dfn:cost}
The \textbf{$|\cdot|_G$-cost} of a graphing $\G = (A_g)_g$ is
\[\C_{|\cdot|_G}(\G) := \sum_{g\in G}|g|_G\cdot \mu(A_g).\]
\end{dfn}

This differs from Levitt's definition by the presence of $|g|_G$ as a weighting factor.

The $|\cdot|_G$-cost will be the basis of several estimates later in the paper.  Simplest among these is the following.

\begin{lem}\label{lem:bound-by-cost}
For every $\eps > 0$ there is a $C_\eps < \infty$ such that for any $T$-graphing $\G$ we have
\[\rmh(\mu,T,\G) \leq \rmH_\mu(A_{e_G}) + C_\eps \cdot \C_{|\cdot|_G}(\G) + \eps.\]
\end{lem}

\begin{proof}
This follows from the bound
\[\rmh(\mu,T,\G) \leq \sum_{g \in G}\rmH_\mu(A_g) = \sum_{g\in G}[-\mu(A_g)\log \mu(A_g) - \mu(X\setminus A_g)\log \mu(X\setminus A_g)]\]
and Lemma~\ref{lem:Fur}.
\end{proof}

Our principal result about graphings and $|\cdot|_G$-cost is the following, which gives us great flexibility in finding low-cost $T$-graphings that are still `large' in the sense of orbit-wise connectedness.

\begin{prop}[Existence of low-cost graphings]\label{prop:lowcost}
Let $G$ be a finitely-generated amenable group of super-linear growth and $(X,\mu,T)$ a free ergodic $G$-system.  Let $U\subseteq X$ have positive measure, and let $\eps > 0$.  Then there is a nontrivial orbit-wise connected $T$-graphing $\G$ such that
\[\Vert(\G) \subseteq U, \quad \mu(\rm{Vert}(\G)) < \eps \quad \hbox{and} \quad \C_{|\cdot|_G}(\G) < \eps.\]
\end{prop}

This proposition will be proved in the next section.

\section{Constructing low-cost graphings}\label{sec:low-cost}

This section culminates in the proof of Proposition~\ref{prop:lowcost}.  First we give two subsections to some preparatory results.  Let $(X,\mu,T)$ be a free ergodic $G$-system.

\subsection{F\o lner sets and skeleta}

The following nomenclature is not standard, but will be useful in the sequel.

\begin{dfn}
Let $(X,d)$ be a metric space and $r > 0$.  An \textbf{$r$-skeleton} of $X$ is a connected graph $(V,E)$ in which $V$ is an $r$-dense subset of $X$ (that is, every element of $X$ lies within distance $r$ of some element of $V$).  Its \textbf{$d$-weight} is the quantity
\[\rm{wt}_d(V,E) = \sum_{uv \in E}d(u,v) \in [0,+\infty].\]
\end{dfn}

\begin{lem}\label{lem:subset-skeleton}
Let $(X,d)$ be a compact metric space, let $r > 0$, and let $(V,E)$ be an $r$-skeleton of $(X,d)$ with $d$-weight $w < \infty$.  Then any subset $Y \subseteq X$ has a $(2r)$-skeleton of $d$-weight at most
\[2w + 2r|V|.\]
\end{lem}

\begin{proof}
Let $W := \{v \in V:\,d(v,Y) < r\}$.  Since $V$ is $r$-dense in the whole of $X$, one must have $B_r(W) \supseteq Y$, where $B_r(W)$ is the union of all open $r$-balls centred at points of $W$.  For each $w \in W$, pick $y_w \in Y\cap B_r(w)$, and let $V_Y:= \{v_w:\,w\in W\}$.  Since $B_{2r}(V_Y) \supseteq B_r(W)$, the set $V_Y$ is $(2r)$-dense in $Y$.

By removing edges from $E$ if necessary, we may assume that it is a spanning tree of $V$.  Then, since $V$ has a spanning tree with $d$-weight $w$, its further subset $W\subseteq V$ has a spanning tree with $d$-weight at most $2w$: this is the classical lower bound of $1/2$ for the Steiner ratio of a general metric space (see, for instance,~\cite[Chapter 3]{Cie01}). Let $E' \subseteq \binom{W}{2}$ be a spanning tree of $W$ with
\[\rm{wt}_d(W,E') = \sum_{ww' \in E'}d(w,w') \leq 2\sum_{xy \in E}d(x,y) = 2w.\]
Let $E_Y := \{v_wv_{w'}:\,ww' \in E'\}$.  Now $(V_Y,E_Y)$ is a $(2r)$-skeleton of $Y$, and
\[\rm{wt}_d(V_Y,E_Y) \leq \rm{wt}_d(W,E') + 2r|E'| \leq 2w + 2r(|W| - 1) \leq 2w + 2r|V|,\]
using the fact that, in a tree such as $(W,E')$, one has $|E'| = |W| - 1$.
\end{proof}

Now let $G$ be a finitely generated amenable group and $d_G$ a right-invariant word metric on it, as before. Given $\eps,r > 0$, let us say that a subset $F \subseteq G$ is \textbf{$(\eps,r)$-F\o lner} if
\[|F| < \infty \quad \hbox{and} \quad |(B_G(r) \cdot F)\setminus F| \leq \eps |F|,\]
where we abbreviate $B_G(e_G,r) =: B_G(r)$.  The amenability of $G$ asserts that $(\eps,r)$-F\o lner sets exist for every $\eps$ and $r$.

The use of two parameters, $\eps$ and $r$, in specifying the F\o lner condition is somewhat redundant, but in some of the proofs that follow it is convenient to be able to manipulate them separately.

We also need our F\o lner sets to satisfy another condition.  Given $E \subseteq G$ and $r > 0$, we say $E$ is \textbf{$r$-connected} if for any $g,h \in E$ there is a finite sequence
\[g = g_0,g_1,\ldots,g_m = h\]
with $g_i \in E$ and $d_G(g_i,g_{i+1}) \leq r$ for every $i = 0,1,\ldots,m-1$.  Such a sequence is called an \textbf{$r$-path}, and the integer $m$ is its \textbf{length}. A set is \textbf{connected} if it is $1$-connected.

\begin{lem}
If $G$ is amenable, then for every $\eps,r > 0$ it has an $(\eps,r)$-F\o lner set which is connected.
\end{lem}

\begin{proof}
\emph{Step 1.}\quad Let $\eta := \eps/|B_G(r)|$. Let $F$ be an $(\eta,r)$-F\o lner set, and let
\[F = F_1\cup \cdots \cup F_k\]
be the partition of $F$ into maximal $(2r)$-connected subsets.  Then we must have
\[i\neq j \quad \Longrightarrow \quad B_G(r)F_i\cap B_G(r)F_j = \emptyset,\]
and therefore
\[\frac{|B_G(r)F\setminus F|}{|F|} = \sum_{i=1}^k\frac{|B_G(r)F_i\setminus F_i|}{|F_i|}\cdot \frac{|F_i|}{|F|}.\]
Since the left-hand side of this equation is at most $\eta$, and the right-hand side is an average weighted by the factors $|F_i|/|F|$, there must be some $i\leq k$ for which
\begin{equation}\label{eq:etaFol}
\frac{|B_G(r)F_i\setminus F_i|}{|F_i|} \leq \eta.
\end{equation}
So $F_i$ is a $(2r)$-connected $(\eta,r)$-F\o lner set.

\vspace{7pt}

\emph{Step 2.}\quad Now let $E := B_G(r) F_i$.  If $g,h \in F_i$ and $d_G(g,h) \leq 2r$, then there is a $1$-path of length at most $2r$ from $g$ to $h$ in $G$, by the definition of the word metric $d_G$.  The first $r$ elements of that path must be contained in $B_G(g,r)$, and the last $r$ elements must be contained in $B_G(h,r)$, so the whole path is contained in $E$.  Since $F_i$ is $(2r)$-connected, it follows that $E$ is connected.

On the other hand, we have
\[B_G(r)\cdot  E = (B_G(r)\cdot F_i) \cup (B_G(r) \cdot (E\setminus F_i)),\]
and the first set in this right-hand union is just $E$ again.  Therefore
\[|(B_G(r)\cdot  E)\setminus E| \leq |B_G(r) \cdot (E\setminus F_i)| \leq |B_G(r)||E\setminus F_i|.\]
By~(\ref{eq:etaFol}), this is at most $\eta|B_G(r)||F_i| \leq \eps|E|$, so $E$ is $(\eps,r)$-F\o lner.
\end{proof}

The main results of this section apply to groups of super-linear growth.  Curiously, their proofs seem to require the following fact from geometric group theory.

\begin{prop}\label{prop:superlin}
If $G$ is a finitely-generated group of super-linear growth, then its growth is at least quadratic: there is a constant $c_1 > 0$ such that
\[|B_G(r)| \geq c_1r^2 \quad \forall r \geq 1.\]
\qed
\end{prop}

Proposition~\ref{prop:superlin} can be deduced by combining Gromov's theorem on groups of polynomial growth with the work of Wolf~\cite{Wol68}, Guivarc'h~\cite{Gui71} and Bass~\cite{Bass72} on the growth of finitely generated nilpotent groups.  However, it also has a more elementary proof: see~\cite[Corollary 3.5]{Man12}. (The latter proof and reference were pointed out to me by Brandon Seward.)

\begin{lem}[Skeleta for F\o lner sets]\label{lem:Fol-skel}
If $G$ is an amenable group of super-linear growth, then there is a constant $c$ with the following property.  For any $r \geq 1$, if $F$ is a connected $(1,r)$-F\o lner set, then it has a $(2r)$-skeleton $(V,E)$ satisfying
\[\rm{wt}_{d_G}(V,E) \leq c|F|/r.\]
\end{lem}

\begin{proof}
Let $c_1$ be the constant given by Proposition~\ref{prop:superlin}.  Let $V \subseteq F$ be a maximal $(2r)$-separated subset, chosen so that it contains $e_G$.  The standard volume-comparison argument gives
\[|V||B_G(r)| = \Big|\bigcup_{g \in V}B_G(g,r)\Big| \leq |B_G(r)\cdot F| \leq 2|F| \quad \Longrightarrow \quad |V| \leq \frac{2|F|}{|B_G(r)|}.\]

Now consider the graph on $V$ in which two points form an edge if the distance between them is at most $5r$.  This graph is connected, by the connectedness of $F$ and the maximality of $V$.  It therefore contains a spanning tree, whose edge-set is a family $E$ of $|V|-1$ pairs of points in $V$.  This gives the bound
\[\rm{wt}_{d_G}(V,E) \leq 5r|E| < 5r|V| \leq \frac{10r|F|}{|B_G(r)|} \leq \frac{10r|F|}{c_1r^2} = \frac{20}{c_1}|F|/r.\]
\end{proof}

The above lemma and Lemma~\ref{lem:subset-skeleton} immediately combine to give the following.

\begin{cor}[Skeleta for subsets of F\o lner sets]\label{cor:Fol-skel}
If $G$ is an amenable group of super-linear growth, then there is a constant $c$ with the following property.  If $r > 0$, $F$ is a connected $(1,r)$-F\o lner set, and $A \subseteq F$, then $A$ has a $(4r)$-skeleton $(V,E)$ satisfying
\[\rm{wt}_{d_G}(V,E) \leq c|F|/r.\]
\qed
\end{cor}

\subsection{Rokhlin subrelations}

We now return to the $G$-system $(X,\mu,T)$.  If $x \in X$ and $A$ is a finite subset of the orbit $T^G(x)$, then we say $A$ is \textbf{$(\eps,r)$-F\o lner} or \textbf{$r$-connected} if this holds for its pre-image in the group: that is, for the set
\[\{g \in G:\ T^gx\in A\}.\]
Since the action is free, this pre-image has the same finite cardinality as $A$.  If we replace $x$ with a different point $T^hx$ in the same orbit, then this pre-image of $A$ changes by right-translation by $h^{-1}$.  This does not affect the properties of being $(\eps,r)$-F\o lner or $r$-connected, so those properties really depend only on the orbit $T^G(x)$ and the set $A$, not on the particular reference point $x$.

\begin{dfn}
Let $\eps,r > 0$.  A subrelation $\R\subseteq \R_T$ is \textbf{$(\eps,r)$-Rokhlin} if it is a Borel equivalence relation, all its equivalence classes are finite and connected, and
\[\mu\{x:\,[x]_\R\ \hbox{is $(\eps,r)$-F\o lner}\} > 1 - \eps.\]
\end{dfn}

This definition has many predecessors in the literature, but usually without requiring connectedness.  That additional demand adapts it to our present needs.

\begin{lem}
If $(X,\mu,T)$ is ergodic and atomless then $\R_T$ has an $(\eps,r)$-Rokhlin subrelation for every $\eps,r > 0$.
\end{lem}

\begin{proof}
According to one of the key results of~\cite{ConFelWei81}, $\R_T$ may be written as $\bigcup_{n\geq 1}\R_n$ for some increasing sequence $\R_1 \subseteq \R_2 \subseteq \cdots \subseteq$ of Borel equivalence relations with finite classes.

For each $i$, define $\R'_i \subseteq \R_i$ by
\[\R_i' = \big\{(x,y)\in \R_i:\,x\ \hbox{and}\ y\ \hbox{lie in the same connected component of}\ [x]_{\R_i}\big\}.\]
These $\R'_i$'s are Borel equivalence relations for which every class $[x]_{\R_i'}$ is finite and connected.  Also, their union is still equal to $\R_T$.  To see this, let $x \in X$ and $g \in G$.  There is a finite $1$-path
\[e= g_0,g_1,\ldots,g_k = g\]
in $G$. Since $\R_T = \bigcup_{n\geq 1} \R_n$, for each $i=0,1,\ldots,k-1$ we have
\[(T^{g_i}x,T^{g_{i+1}}x) \in \R_n \quad \hbox{for all sufficiently large $n$},\]
and therefore in fact
\[(T^{g_i}x,T^{g_{i+1}}x) \in \R'_n \quad \hbox{for all sufficiently large $n$},\]
since $d_G(g_i,g_{i+1}) = 1$ for each $i$. Hence, by transitivity, $(x,T^gx) \in \R'_n$ for all sufficiently large $n$.

Finally,
\begin{align*}
\int \frac{|T^{B_G(r)}([x]_{\R_n'})\setminus [x]_{\R_n'}|}{|[x]_{\R_n'}|}\,\mu(dx) &\leq \sum_{g \in B_G(r)}\int \frac{|T^g([x]_{\R_n'})\setminus [x]_{\R_n'}|}{|[x]_{\R_n'}|}\,\mu(dx)\\
&= \sum_{g \in B_G(r)}\int \frac{|\{y \in [x]_{\R_n'}:\ T^gy \not\in [x]_{\R_n'}\}|}{|[x]_{\R_n'}|}\,\mu(dx)\\
&= \sum_{g \in B_G(r)}\mu\{x:\,(x,T^gx)\not\in \R_n'\}.
\end{align*}
This tends to $0$ as $n\to\infty$ because $\R_T = \bigcup_{n\geq 1}\R_n'$.  By Chebychev's inequality, this implies that
\[\mu\{x:\,[x]_{\R_n'}\ \hbox{is $(\eps,r)$-F\o lner}\} > 1-\eps\]
for all sufficiently large $n$.
\end{proof}

\subsection{Existence of low-cost graphings}

We are ready to prove Proposition~\ref{prop:lowcost}. The required $T$-graphing will be built as a union of a sequence of $T$-graphings given by the following lemma.  Given two measurable subsets $U,V \subseteq X$, we say that $V$ is $(T,r)$-\textbf{dense} in $U$ if
\[T^{B_G(r)}V \supseteq U.\]

\begin{lem}\label{lem:lowcost}
If $G$ is an amenable group of super-linear growth, then there is a constant $c < \infty$ with the following property. Let $(X,\mu,T)$ be a free $G$-system and let $U\subseteq X$ have positive measure.  Let $0 < \eps < \mu(U)$ and $r < \infty$.  If $\R\subseteq \R_T$ is an $(\eps,r)$-Rokhlin subrelation, then there is a $T$-graphing $\G = (A_g)_g$ with the following properties:
\begin{itemize}
\item[i)] $V := \Vert(\G) \subseteq U$,
\item[ii)] $\C_{|\cdot|_G}(\G) \leq c/r$,
\item[iii)] $\R \cap (V\times V)= \R_\G$,
\item [iv)] $V$ is $(T,4r)$-dense in the set
\[U \cap \{x:\,[x]_\R\ \hbox{is $(1,r)$-F\o lner}\}\]
(in particular, this implies that $\mu(V) > 0$).
\end{itemize}
\end{lem}

\begin{proof}
Let $c$ be the constant from Corollary~\ref{cor:Fol-skel}. Suppose that $\R$ is an $(\eps,r)$-Rokhlin subrelation, and let
\[ X_0 := \{x:\,[x]_\R\ \hbox{is $(1,r)$-F\o lner}\},\]
so $X_0$ is a union of $\R$-classes and $\mu(X_0) > 1 - \eps$.

Since all classes in $\R$ are finite, it has a transversal $Y \subseteq X$: that is, $Y$ is measurable and contains a unique element from each class of $\R$ (see, for instance,~\cite[Example 6.1]{KecMil04}). For each $x \in X_0$ let us write $\ol{x}$ for the unique element of $[x]_\R\cap Y$.

For $y \in Y\cap X_0$, let
\[B_y := \{g:\, T^gy \in U \cap [y]_\R\} \subseteq G.\]
Since $y \in X_0$ and $B_y$ is contained in $\{g:\ T^gy \in [y]_\R\}$, Corollary~\ref{cor:Fol-skel} gives a $(4r)$-skeleton $(W^0_y,E^0_y)$ for $B_y$ satisfying
\[\rm{wt}_{d_G}(W^0_y,E^0_y) \leq c|[y]_\R|/r.\]
Clearly $W^0_y$ and $E^0_y$ may be chosen measurably in $y$.

Now transport these skeleta from $G$ back to $X$ by setting
\[W_y:= T^{W^0_y}(y) \quad \hbox{and} \quad E_y := \big\{\{T^hy,T^gy\}:\ \{h,g\} \in E^0_y\big\} \quad \hbox{for}\ y \in Y\cap X_0.\]
The result is a graph $(W_y,E_y)$ on a subset of each class $[y]_\R \subseteq X_0$.

Finally, define the $T$-graphing $\G = (A_g)_g$ by setting
\begin{multline*}
A_{e_G} := \{x \in X_0:\ x \in W_{\ol{x}}\}\\ \hbox{and} \quad A_g := \{x \in X_0: x \in W_{\ol{x}}\ \hbox{and}\ \{x,T^gx\} \in E_{\ol{x}}\} \quad \hbox{for}\ g \in G\setminus \{e_G\}.
\end{multline*}
This is symmetric: if $x \in A_g$ and we set $x' := T^gx$, then $\{x,T^gx\} = \{x',T^{g^{-1}}x'\}\in E_{\ol{x}}$ and so also $x' \in W_{\ol{x}}$ and $x' \in A_{g^{-1}}$.

It remains to verify the four required properties.

\begin{itemize}
\item[i)] For each $x \in X_0$ we have
\[W_{\ol{x}} = T^{W^0_{\ol{x}}}(\ol{x}) \subseteq T^{B_{\ol{x}}}(\ol{x}) \subseteq U,\]
by the definition of $B_{\ol{x}}$. Hence $A_g \subseteq U$ for each $g$.
\item[ii)] To estimate the cost, first observe that we may write
\[A_g = \bigcup_{h \in G}\big\{T^hy:\ y \in Y\cap X_0\ \hbox{and}\ \{h,gh\} \in E^0_y\big\}.\]
This is a disjoint union: if $T^hy = T^{h'}y'$ among the points allowed above, then this point lies in $W_y \subseteq [y]_\R$ by the definition of $E_y^0$, and this implies that $h = h'$ and $y = y'$ because $T$ is free and $Y$ contains a unique element in each class of $\R$. Therefore
\begin{align*}
\C_{|\cdot|_G}(\G) &= \sum_{g\in G}|g|_G\cdot \mu(A_g)\\
&= \sum_{g\in G}\sum_{h\in G}|g|_G\cdot \mu\{y \in Y\cap X_0:\,\{h,gh\} \in E^0_y\}\\
&= \int_{Y\cap X_0} \sum_{\{h,gh\} \in E_y^0} d_G(h,gh)\ \mu(d y)\\
&= \int_{Y\cap X_0}\rm{wt}_{d_G}(W^0_y,E^0_y)\,\mu(dy)\\ &\leq \frac{c}{r} \int_{Y\cap X_0}|[y]_\R|\ \mu(dy) \leq \frac{c}{r}\int_Y|[y]_\R|\ \mu(dy) = \frac{c}{r}.
\end{align*}
\item[iii)]   Observe that
\[V = \bigcup_g A_g = \bigcup_{y \in Y\cap X_0} W_y.\]
Therefore
\[\R \cap(V\times V) = \bigcup_{y \in Y\cap X_0}W_y\times W_y,\]
and this equals $\R_\G$ because all the graphs $(W_y,E_y)$ are connected.

\item[iv)] Lastly, if $x \in U \cap X_0$, then $B_{\ol{x}}$ is nonempty, and then $W^0_{\ol{x}}$ is $(4r)$-dense in $B_{\ol{x}}$ by construction.  This implies that $V$ is $(T,4r)$-dense in $U\cap X_0$.
\end{itemize}
\end{proof}

\begin{proof}[Proof of Proposition~\ref{prop:lowcost}]
Let $c$ be the constant from Lemma~\ref{lem:lowcost}, and choose $m \in \bbN$ so that $2^{-m} < \eps/2c$.  Also, shrink $U$ if necessary so that $0 < \mu(U) < \eps$.

\vspace{7pt}

\emph{Step 1.}\quad For each $n \geq m$, let $\R_n \subseteq \R_T$ be a $(2^{-n},2^n)$-Rokhlin subrelation which also has the property that the set
\[C_n := \big\{x:\,[x]_{\R_n}\ \hbox{is $(1,2^n)$-F\o lner and}\ T^{B_G(2^{n+3})}x\subseteq [x]_{\R_n}\big\}\]
satisfies $\mu(C_n) > 1 - 2^{-n-1}\mu(U)$.

Now let
\[W := U\cap \bigcap_{n\geq m} C_n,\]
so $\mu(W) > \mu(U)/2$.

\vspace{7pt}

\emph{Step 2.}\quad Applying Lemma~\ref{lem:lowcost}, let $\G_m = (A_{m,g})_g$ be a $T$-graphing such that $V := \Vert(\G_m) \subseteq W$ has $\mu(V) > 0$,
\[\C_{|\cdot|_G}(\G_m) \leq c/2^m,\]
and
\[\R_m \cap (V \times V) = \R_{\G_m}.\]

\vspace{7pt}

\emph{Step 3.}\quad For each $n\geq m+1$ now apply Lemma~\ref{lem:lowcost} again to obtain a $T$-graphing $\G_n = (A_{n,g})_g$ such that $V_n \subseteq V$,
\[\C_{|\cdot|_G}(\G_n) < c/2^n,\]
\[\R_n\cap (V_n\times V_n) = \R_{\G_n},\]
and $V_n$ is $(T,2^{n+2})$-dense in $V$. The last conclusion can be obtained from part (iv) of Lemma~\ref{lem:lowcost} because $V$ is contained in $W$ and $W$ is already contained in $C_n$ by construction. Observe that the choices of $\G_n$ for $n > m$ depend on the choice of $\G_m$ in Step 2, but not on each other.

\vspace{7pt}

\emph{Step 4.}\quad After finishing this recursion, define $\G = (A_g)_g$ by
\[A_g:= \bigcup_{n\geq m}A_{n,g} \quad \hbox{for each}\ g \in G.\]
We will show that this has the desired properties.  The symmetry property~(\ref{eq:symmetry}) holds for $\G$ because it holds for each $\G_n$. The vertex set of $\G$ is equal to $V \subseteq U$ because $\Vert(\G_m) = V$ and then $\Vert(\G_n) \subseteq V$ for every $n\geq m+1$.  This implies that $\mu(\rm{Vert}(\G)) < \eps$, because $\mu(U)  < \eps$.  Also,
\[\C_{|\cdot|_G}(\G) \leq \sum_{n\geq m}\C_{|\cdot|_G}(\G_n) < \sum_{n\geq m}c2^{-n} = c2^{-m+1} < \eps.\]

It remains to show that $\G$ is orbit-wise connected: that is, that
\[\R_\G = \R_T\cap (V\times V).\]
Consider a pair of distinct points $x,T^gx \in V$.  Choose the least $n\geq m$ which satisfies
\[2^{n+3} > |g|_G.\]
We will prove that $(x,T^gx) \in \R_\G$ by induction on this $n$.

If $n=m$, then $|g|_G < 2^{m + 3}$, and so certainly
\[T^{B_G(2^{m+3})}x \cap T^{B_G(2^{m+3})}(T^gx) \neq \emptyset.\]
Since also $x,T^gx \in C_m$, the definition of $C_m$ now requires that $[x]_{\R_m} = [T^gx]_{\R_m}$.  Since $\R_m\cap (V\times V) = \R_{\G_m}$, it follows that $(x,T^gx) \in \R_{\G_m} \subseteq \R_\G$.

Now suppose that $n\geq m+1$.  Since $V_n$ is $(T,2^{n+2})$-dense in $V$, there are $h,k \in G$ such that $T^hx,T^kx \in V_n$ and
\[d_G(e_G,h),d_G(g,k) < 2^{n+2}.\]
By the inductive hypothesis, this implies that $(x,T^hx),(T^gx,T^kx) \in \R_\G$.  On the other hand, the triangle inequality now gives $d_G(h,k) < 2^{n+4}$, and so
\[T^{B_G(2^{n+3})}(T^hx)\cap T^{B_G(2^{n+3})}(T^kx) \neq \emptyset.\]
Arguing as above, this implies that $(T^hx,T^kx) \in \R_n$, and since these points are also in $V_n$ this implies that
\[(T^hx,T^kx) \in \R_n\cap (V_n\times V_n) = \R_{\G_n} \subseteq \R_\G.\]
\end{proof}

\begin{rmk}
It is easy to see that the above conclusion fails if $G = \bbZ$, so the assumption of super-linear growth is important. \fin
\end{rmk}

\section{Proof of the derandomization results}\label{sec:BandC}

\begin{prop}\label{prop:derandomize}
Let $(\s,U)$ be a partial cocycle for which at least one of the following holds:
\begin{itemize}
\item[i)] $\s$ extends to an integrable cocycle $G\times X\to H$;
\item[ii)] $(\s,U)$ is bounded.
\end{itemize}

Then for every $\eps > 0$ there is a $\delta > 0$ such that the following holds.  If $\G = (A_g)_{g\in G}$ is an orbit-wise connected $T$-graphing with $\rm{Vert}(\G) \subseteq U$, and both $\mu(A_{e_G}) < \delta$ and $\C_{|\cdot|_G}(\G) < \delta$, then $\rmh(\mu,T,\s|\G) < \eps$.
\end{prop}

\begin{proof}
\emph{Case (i).}\quad Denote the extended cocycle $G\times X\to H$ also by $\s$.  From $\G$ we define a nearest-neighbour $T$-graphing $\Theta = (B_s)_{s \in B_G(1)}$ as follows.  For each $g \in G$, choose a word of length $|g|_G$ in the alphabet $B_G(1)$ which evaluates to $g$, say
\[g = s_{g,n}s_{g,n-1}\cdots s_{g,1} \quad \hbox{where}\ n := |g|_G.\]
Now define
\[B_{g,i} := T^{s_{g,i-1}\cdots s_{g,1}}A_g \quad \hbox{for}\ i = 1,2,\ldots,n,\]
and finally let
\[B_s := \bigcup_{g \in G}\ \ \bigcup_{1 \leq i\leq |g|
_G\ \rm{s.t.}\ s_{g,i} = s}B_{g,i} \quad \hbox{for each}\ s \in B_G(1).\]

This construction has the following two important features.
\begin{itemize}
 \item[(a)] The new cost is bounded by the old cost:
\begin{multline*}
\C_{|\cdot|_G}(\Theta) = \sum_{s \in B_G(1)}\mu(B_s) \leq \sum_{g\in G}\sum_{i =1}^{|g|_G}\mu(B_{g,i})\\ = \sum_{g\in G}\sum_{i=1}^{|g|_G}\mu(A_g) = \sum_{g\in G}|g|_G\cdot \mu(A_g) = \C_{|\cdot|_G}(\G) < \delta.
\end{multline*}
\item[(b)] Given the collection of sets $A_g$ for $g \in G$ and also the collection of partial functions $\s(s,\cdot)|B_s$ for $s \in B_G(1)$, they determine all the partial functions $\s(g,\cdot)|A_g$ using the cocycle identity:
\[\s(g,x) = \s(s_{g,n},T^{s_{g,n-1}\cdots s_{g,1}}x)\cdots \s(s_{g,1},x).\]
In this formula, if $x \in A_g$, then
\[x \in B_{g,1} \subseteq B_{s_{g,1}},\quad T^{s_{g,1}}x \in B_{g,2} \subseteq B_{s_{g,2}},\quad \dots,\quad T^{s_{g,n-1}\cdots s_{g,1}}x \in  B_{s_{g,n-1}}.\]
Therefore the factor generated by $\s|\G$ is contained in the factor generated by $\G$ and $\s|\Theta$ together.
\end{itemize}

By property (b) above, we have
\[\rmh(\mu,T,\s|\G) \leq \rmh(\mu,T,\G) + \rmh(\mu,T,\s|\Theta)\]
If $\delta$ is sufficiently small, then the first of these terms is at most $\eps/2$ by Lemma~\ref{lem:bound-by-cost}.  On the other hand, Lemma~\ref{lem:Shannon-bound2} gives
\[\rmh(\mu,T,\s|\Theta) \leq \sum_{s \in B_G(1)}\rmH_\mu(\s(s,\,\cdot\,);B_s).\]
By property (a) above, if $\delta$ is small enough, then we may apply Corollary~\ref{cor:cocyc-part-ent} to each summand on the right.  This completes the proof in case (i).

\vspace{7pt}

\emph{Case (ii).}\quad In this case we can use Lemma~\ref{lem:Shannon-bound2} more directly:
\begin{equation}\label{eq:h-est}
\rmh(\mu,T,\s|\G) \leq \sum_{g\in G}\rmH_\mu(A_g) + \sum_{g\in G}\mu(A_g)\rmH_{\mu_{|A_g}}(\s(g,\,\cdot\,)).
\end{equation}
Since we are in case (ii), there is a finite constant $C$ such that for each $g$ the random variable $\s(g,\,\cdot\,)$ takes values in $B_H(C|g|_G)$ almost surely.  Since $H$ is finitely generated, we have $\log|B_H(r)| = O(r)$ for all $r$, and hence
\[\rmH_{\mu|_{A_g}}(\s(g,\,\cdot\,)) \leq \log |B_H(C|g|_G)| = O(C|g|_G) = O(|g|_G).\]
Therefore the right-hand side of~(\ref{eq:h-est}) is bounded by a constant multiple of
\[\sum_{g\in G}\rmH_\mu(A_g) + \sum_{g\in G}|g|_G\cdot\mu(A_g).\]
By Lemma~\ref{lem:bound-by-cost}, the first term here may also be made arbitrarily small if $\mu(A_{e_G})$ and $\C_{|\cdot|_G}(\G)$ are sufficiently small. This completes the proof.
\end{proof}

\begin{proof}[Proof of Theorem C]
	Given $\eps > 0$, apply case (i) of Proposition~\ref{prop:derandomize} to the cocycle $\s$ with $U := X$.  We obtain a $\delta > 0$ for which the conclusion of that proposition holds. Now apply Proposition~\ref{prop:lowcost} to obtain a nontrivial orbit-wise connected $T$-graphing $\G$ such that $\rm{Vert}(\G) \subseteq U$, $\mu(\rm{Vert}(\G)) < \delta$ and $\cal{C}_{|\cdot|_G}(\G) < \delta$.  The second of these conditions implies that also $\mu(A_{e_G}) < \delta$.  Therefore, letting $\A$ be the factor generated by $\s|\G$, the choice of $\delta$ implies that $\rmh(\mu,T,\A) < \eps$.
	
	Let $V := \rm{Vert}(\G)$.  Since $\G$ is orbit-wise connected, Lemma~\ref{lem:cocyc-gen} tells us that the partial cocycle $(\s_{|V},V)$ is also $\A$-measurable. Now apply the first part of Proposition~\ref{prop:cocyc-extend} to the partial cocycle $(\s_{|V},V)$ and the factor system of $(X,\mu,T)$ generated by $\A$, which must still be ergodic.  That proposition gives an $\A$-measurable cocycle $\tau:G\times X\to H$ such that $\s_{|V} = \tau_{|V}$.  Since $\tau$ is $\A$-measurable, its entropy is also less than $\eps$, and by the second part of Proposition~\ref{prop:cocyc-extend} it is cohomologous to $\s$.
\end{proof}

\begin{proof}[Proof of Theorem D]
Fix $\eps > 0$.  Cases (i) and (ii) of Theorem D correspond to cases (i) and (ii) of Proposition~\ref{prop:derandomize}.  Therefore in either case there is some $\delta > 0$ for which the implication of that proposition holds. Having chosen this $\delta$, Proposition~\ref{prop:lowcost} gives a non-trivial orbit-wise connected $T$-graphing  $\G = (A_g)_{g \in G}$ such that
\[U := \rm{Vert}(\G) \subseteq \dom\Phi, \quad \mu(\rm{Vert}(\G)) < \delta \quad \hbox{and} \quad \C_{|\cdot|_G}(\G) < \delta.\]
By the choice of $\delta$ this implies that $\rmh(\mu,T,\a|\G) < \eps$.

Letting $\A$ be the factor generated by $(\a_{|U},U)$, it now follows by Lemma~\ref{lem:cocyc-gen} that $\rmh(\mu,T,\A) < \eps$.  By enlarging $\A$ slightly if necessary, we may assume in addition that it is generated by a factor map to another free $G$-system.  Finally Lemma~\ref{lem:pre-C} produces the remaining objects with the properties asserted in Theorem D.
\end{proof}

\section{Further questions}

Integrable measure equivalence was originally introduced in~\cite{BadFurSau13} for actions of lattices in isometry groups of hyperbolic spaces.  It would be interesting to know whether any classification of probability-preserving actions of such groups follows from the accompanying assumption of SSOE$_1$.  Since these groups are not amenable, the Kolmogorov--Sinai entropy is not available as an invariant.  However, recent years have seen important progress in our understanding of entropy-like invariants for non-amenable groups.

\begin{ques}
	If $G$ and $H$ are countable groups, does an SOE$_\infty$ or SSOE$_1$ between a $G$-action and an $H$-action imply a relation between their Rokhlin entropies~\cite{Seward--KriI,Seward--KriII}?
\end{ques}

\begin{ques}
	If $G$ and $H$ are sofic groups, can one choose sofic approximations to them in such a way that an SOE$_\infty$ or SSOE$_1$ between a $G$-action and an $H$-action imply a relation between their sofic entropies~\cite{Bowen10,KerLi11a}?  If $G = H$ is a free group, can one obtain a relation between f-invariants~\cite{Bowen10free}?
\end{ques}

\bibliographystyle{alpha}
\bibliography{bibfile}

\vspace{7pt}

\noindent\small{\textsc{Tim Austin}\\ \textsc{Courant Institute of Mathematical Sciences, New York University, 251 Mercer St, New York NY 10012, USA}

\vspace{7pt}

\noindent Email: \verb|tim@cims.nyu.edu|

\noindent URL: \verb|cims.nyu.edu/~tim|}

\end{document}